\documentclass[11pt,a4paper,reqno]{amsart}
\usepackage[T1]{fontenc}
\usepackage[english]{babel}
\usepackage{mathtools}
\usepackage{amssymb}
\usepackage{xcolor}
\IfFileExists{libertinus.sty}{\usepackage{libertinus}}{\usepackage{lmodern}}
\IfFileExists{BOONDOX-r-calo.tfm}{\usepackage[cal=boondoxo,bb=ams]{mathalfa}}{}
\usepackage{microtype}
\usepackage{enumitem}
\usepackage{etoolbox}
\usepackage{tikz}
\usetikzlibrary{decorations.pathreplacing}
\usepackage[nosort,nocompress]{cite}

\allowdisplaybreaks[2]
\usepackage[
 a4paper,
 left=30mm,
 right=30mm,
 top=27mm,
 bottom=30mm,
 headsep=8mm,
 footskip=13mm,
 heightrounded
]{geometry}

\setlist[itemize]{leftmargin=2.1em,itemsep=0.25em,topsep=0.45em,parsep=0pt}
\numberwithin{equation}{section}

\makeatletter
\def\@settitle{%
 \begin{center}
  \vspace*{-0.4em}
  {\normalfont\bfseries\fontsize{17}{21}\selectfont\@title\par}
  \vspace{0.7em}
 \end{center}}
\renewcommand\section{\@startsection{section}{1}{\z@}%
 {1.5\baselineskip plus 0.2\baselineskip minus 0.1\baselineskip}%
 {0.65\baselineskip}{\normalfont\Large\bfseries}}
\renewcommand\subsection{\@startsection{subsection}{2}{\z@}%
 {1.15\baselineskip plus 0.2\baselineskip minus 0.1\baselineskip}%
 {0.45\baselineskip}{\normalfont\large\bfseries}}
\makeatother

\makeatletter
\patchcmd{\@setauthors}{\centering\footnotesize}{\centering\normalsize}{}{}
\patchcmd{\@setauthors}{\MakeUppercase{\authors}}{\authors}{}{}
\patchcmd{\maketitle}{\uppercasenonmath\shorttitle}{}{}{}
\patchcmd{\maketitle}{\@nx\MakeUppercase{\the\toks@}}{\the\toks@}{}{}
\makeatother

\usepackage{hyperref}
\hypersetup{
 hidelinks,
 pdfencoding=auto,
 pdftitle={Sharp L2 estimates for wave equations with non-integrable speeds},
 pdfauthor={Halit Sevki Aslan, Wenhui Chen},
 pdfsubject={Large time behavior of wave equations with time-dependent propagation speeds},
 pdfkeywords={wave equation, time-dependent coefficient, sharp L2 estimate, low-frequency analysis}
}

\theoremstyle{plain}
\newtheorem{theorem}{Theorem}[section]
\newtheorem{lemma}[theorem]{Lemma}

\theoremstyle{definition}
\newtheorem{assumption}{Assumption}[section]

\theoremstyle{remark}
\newtheorem{remark}[theorem]{Remark}

\newcommand{\ml}{\mathcal}
\newcommand{\mb}{\mathbb}

\title[Wave equations with non-integrable propagation speeds]{Sharp $L^2$ estimates for wave equations with non-integrable speeds}

\author[H. S. Aslan]{Halit Sevki Aslan}
\address{Department of Computing and Mathematics, University of S\~ao Paulo, Ribeir\~ao Preto 14040-901, SP, Brazil}
\email{halitsevkiaslan@gmail.com}

\author[W. Chen]{Wenhui Chen}
\address{School of Mathematics and Information Science, Guangzhou University, Guangzhou 510006, P. R. China}
\email{wenhui.chen.math@gmail.com}

\date{}

\begin{document}

\begin{abstract}
We study the large-time behavior of the displacement for wave equations with non-integrable speeds and zero initial displacement. For initial velocities in $L^2\cap L^1$, we derive an $L^2$ upper bound under a weighted Poincar\'e inequality and symbol estimates up to order two on the propagation speed. The rate is determined by the speed and its time integral. For these speeds, an additional symbol estimate for the third derivative and the nonzero mean condition on the initial velocity give a lower bound of the same order, obtained without using an explicit formula for the Fourier multiplier. At the endpoint $a(t)=(1+t)^{-1}$, the sharp rate is $t(\ln t)^{-\frac{n}{4}}$. The estimates apply to polynomial, logarithmically modified, oscillating, and exponential-type speeds.
\end{abstract}

\maketitle

\section{Introduction}
We derive sharp large-time $L^2$ estimates for the solution $u=u(t,x)$ to the Cauchy problem
\begin{align}\label{eq:main-problem}
\begin{cases}
u_{tt}-[a(t)]^2\Delta u=0,&x\in\mb R^n,\ t>0,\\
u(0,x)=0,\ \ u_t(0,x)=u_1(x),&x\in\mb R^n,
\end{cases}
\end{align}
where $n\geqslant1$ and $a\in\ml C^1\bigl([0,+\infty)\bigr)$ is positive.  Throughout the paper the propagation speed is non-integrable, namely, the optical distance
\begin{align}\label{eq:optical-distance}
\Lambda(t):=1+\int_0^t a(s)\,\mathrm ds\to+\infty \ \ \mbox{as}\ \ t\to+\infty.
\end{align}
The additive constant ensures that $\Lambda(t)\geqslant1$, so its reciprocal powers are defined also at $t=0$.

The $L^2$ norm of the displacement is not directly controlled by the natural energy, which is
\begin{align*}
E_u(t):=\frac{1}{2}\left(\|u_t(t,\cdot)\|_{L^2}^2+[a(t)]^2\|\nabla u(t,\cdot)\|_{L^2}^2\right)
\end{align*}
and satisfies
\begin{align*}
E_u'(t)=a(t)a'(t)\|\nabla u(t,\cdot)\|_{L^2}^2.
\end{align*}
For a constant speed, the displacement of a nontrivial solution with $u(0,x)=0$ and $u_1\in L^2\cap L^1$ does not tend to zero in $L^2$ as $t\to+\infty$. This follows by restricting the Plancherel integral to a fixed annulus with positive Fourier mass and applying the Riemann--Lebesgue lemma. For time-dependent speeds, however, $L^2$ decay is possible. For $a(t)=\mathrm e^t$ and $u_1\in L^2\cap L^1$, Theorem~\ref{thm:upper} gives $\|u(t,\cdot)\|_{L^2}\to0$ as $t\to+\infty$ in every dimension, while $E_u$ is nondecreasing. The corresponding sharp rates are given in \eqref{eq:stretched-exponential-rate}.

Local energy decay for constant-speed wave equations in exterior domains was studied in \cite{Morawetz=1961,Morawetz-Ralston-Strauss=1977}. For scattering theory and related background, we refer to \cite{Lax-Phillips=1989,Ebert-Reissig=2018}. Whole-space $L^p$--$L^q$ estimates were obtained in \cite{Strichartz=1970,Peral=1980}.  For the global $L^2$ norm of the displacement, the behavior of the Fourier multiplier for small frequencies plays an essential role.  For $u_1\in L^2\cap L^{1,1}$ with nonzero mean
\begin{align*}
P_{u_1}:=\int_{\mb R^n}u_1(x)\,\mathrm dx\neq0,
\end{align*}
the sharp estimates in \cite{Ikehata=2023,Chen-Ikehata=2026} are
\begin{align*}
\|u(t,\cdot)\|_{L^2} \approx
\begin{cases}
t^{\frac{1}{2}}&\text{if}\ \ n=1,\\
(\ln t)^{\frac{1}{2}}&\text{if}\ \ n=2,\\
1&\text{if}\ \ n\geqslant3,
\end{cases}
\end{align*}
as $t\to+\infty$.  The weighted $L^{1,1}$ assumption was removed for $n\leqslant2$ in \cite[Appendix~A]{Chen-Takeda=2023} and for every $n\geqslant1$ in \cite[Appendix: wave equation]{Takeda=2026}, with $u_1\in L^2\cap L^1$.  For the free wave equation with general initial data $(u_0,u_1)\in L^2\times H^{-1}$, \cite{Chen=2026-Ball} proved that $\|u(t,\cdot)\|_{L^2}$ has a finite limit if $u_1\in\dot H^{-1}$ and tends to infinity if $u_1\in H^{-1}\setminus\dot H^{-1}$.

For variable propagation speeds, most results concern the natural energy or dispersive upper estimates.  Energy asymptotics and generalized energy conservation were studied in \cite{Hirosawa=2007,Hirosawa-Wirth=2009,Ebert-Fitriana-Hirosawa=2015,Ghisi-Gobbino=2025}. Earlier $L^p$--$L^q$ estimates for a special weakly hyperbolic equation were obtained in \cite{Reissig=1997}. The $L^p$--$L^q$ upper estimates were derived in \cite{Reissig-Yagdjian=2000-increasing,Reissig-Yagdjian=2000-oscillations,Galstian=2003,Reissig-Smith=2005,Aslan-Reissig=2022} for several classes of increasing, bounded, and oscillating coefficients.  Fourier representations based on phase-space decompositions and WKB analysis are discussed in \cite{Reissig=2004}. Scattering and asymptotic profiles were considered in \cite{Matsuyama=2006}.  In one space dimension, \cite{Ikehata=2025-wave} proved $L^2$ boundedness for smooth compactly supported data with $P_{u_1}=0$ and positive bounded speeds bounded away from zero. The result covers monotone speeds and also speeds with $a'\in L^1(0,+\infty)$.  To the best of our knowledge, sharp upper and lower estimates for the displacement have not been established in a unified framework for the class of non-integrable propagation speeds considered below.

The large-time rate depends on frequencies that remain near the origin and on those that have left this region. To describe the latter contribution in terms of $a$ and $\Lambda$, we introduce the integral
\begin{align}\label{eq:general-Q}
\ml Q_{a,n}(t):=\int_0^t\frac{(1+s)^2[a(s)]^2}{[\Lambda(s)]^{n+1}}\,\mathrm ds
\end{align}
and define the resulting rate function by
\begin{align}\label{eq:general-rate}
\ml G_{a,n}(t):=\left[t^2[\Lambda(t)]^{-n}+\frac{\ml Q_{a,n}(t)}{a(t)}\right]^{\frac{1}{2}}.
\end{align}
The terms $t^2[\Lambda(t)]^{-n}$ and $\frac{\ml Q_{a,n}(t)}{a(t)}$ in $[\ml G_{a,n}(t)]^2$ correspond to these two contributions, respectively. Their origin is explained in Section~\ref{sec:sharp-rate}.
We use the following two assumptions on the propagation speed. The first is a weighted Poincar\'e inequality on $(0,T)$, with constant bounded by a fixed multiple of $[\Lambda(T)]^2$ uniformly in $T$.

\begin{assumption}
There exists a constant $C>0$ such that, for every $T>0$ and $\varphi\in H_0^1(0,T)$,
\begin{align}\label{eq:A1}
\int_0^T[a(s)]^2|\varphi(s)|^2\,\mathrm ds \leqslant C[\Lambda(T)]^2 \int_0^T|\varphi'(s)|^2\,\mathrm ds. \tag{$\mathrm A_1$}
\end{align}
\end{assumption}

\begin{assumption}
For an integer $m\geqslant1$, let $a\in\ml C^m\bigl([0,+\infty)\bigr)$ satisfy
\begin{align}\label{eq:A2m}
|a^{(k)}(t)| \lesssim a(t)\left(\frac{a(t)}{\Lambda(t)}\right)^k \ \ \mbox{for}\ \ k\in\{1,\ldots,m\}\ \ \mbox{and}\ \ t\geqslant0. \tag{$\mathrm A_{2,m}$}
\end{align}
\end{assumption}

For a positive speed $a\in\ml C^2\bigl([0,+\infty)\bigr)$ satisfying \eqref{eq:A1} and \eqref{eq:A2m} with $m=2$, Theorem~\ref{thm:upper} gives the upper bound with rate $\ml G_{a,n}$ for each $u_1\in L^2\cap L^1$. If, in addition, $a\in\ml C^3\bigl([0,+\infty)\bigr)$ and the third-order bound in \eqref{eq:A2m} holds, Theorem~\ref{thm:sharp} gives a matching lower bound when $P_{u_1}\neq0$. Under these assumptions, for all sufficiently large $t$, we obtain
\begin{align*}
\ml G_{a,n}(t)|P_{u_1}| \lesssim \|u(t,\cdot)\|_{L^2} \lesssim \ml G_{a,n}(t)\|u_1\|_{L^2\cap L^1}.
\end{align*}
The lower-bound constant and the starting time may depend on $u_1$, as specified in Theorem~\ref{thm:sharp}.

The function $\ml G_{a,n}$ can be evaluated for several concrete classes of propagation speeds.  For polynomial and logarithmically modified speeds, the sharp estimates determine the transitions among growth, boundedness, and decay. In the constant-speed case, they recover the classical rates recalled above from \cite{Ikehata=2023,Chen-Ikehata=2026}. The estimates also cover oscillating speeds that are not comparable to any positive monotone function, as illustrated by \eqref{eq:nonmonotone-speed} below.  For exponential-type speeds, they yield decay for every $n\geqslant1$.

\medskip \paragraph{Notation.} Unless otherwise stated, positive constants independent of $t$ and the frequency are denoted by $c$ and $C$ and may change from line to line.  They may depend on $n$ and on the fixed speed $a$.  We write $f\lesssim g$ if $f\leqslant Cg$, and $f\approx g$ if both $f\lesssim g$ and $g\lesssim f$. For vectors and matrices, $|\cdot|$ and $\|\cdot\|$ denote the Euclidean norm and the induced operator norm, respectively. Subscripts on constants or comparison symbols indicate dependence on the corresponding parameters. We use the weighted space
\begin{align*}
L^{1,1}:=\left\{f\in L^1:\ \int_{\mb R^n}(1+|x|)|f(x)|\,\mathrm dx<+\infty\right\}.
\end{align*}

\section{Main results}

We consider the solution to \eqref{eq:main-problem} with zero initial displacement and initial velocity $u_1$. For general initial data $u_0,u_1\in L^1$, the zero-frequency mode satisfies
\begin{align*}
\widehat u(t,0)=\widehat u_0(0)+t\widehat u_1(0).
\end{align*}
Thus, when $P_{u_1}\neq0$, the initial velocity determines the leading zero-frequency behavior.

\subsection{Sharp \texorpdfstring{$L^2$}{L2} estimates}

\begin{theorem}[Upper estimate]\label{thm:upper}
Let $a\in\ml C^2\bigl([0,+\infty)\bigr)$ be positive and satisfy \eqref{eq:A1} and \eqref{eq:A2m} with $m=2$. If $u_1\in L^2\cap L^1$, then the solution to the linear wave equation \eqref{eq:main-problem} satisfies
\begin{align}\label{eq:general-upper}
\|u(t,\cdot)\|_{L^2} \lesssim \ml G_{a,n}(t)\|u_1\|_{L^2\cap L^1}
\end{align}
for all sufficiently large $t$.
\end{theorem}

\begin{theorem}[Sharp estimate]\label{thm:sharp}
Let $a\in\ml C^3\bigl([0,+\infty)\bigr)$ be positive and satisfy \eqref{eq:A1} and \eqref{eq:A2m} with $m=3$. If $u_1\in L^2\cap L^1$ and $P_{u_1}\neq0$, then the solution to the linear wave equation \eqref{eq:main-problem} satisfies
\begin{align}\label{eq:general-sharp}
\ml G_{a,n}(t)|P_{u_1}| \lesssim \|u(t,\cdot)\|_{L^2} \lesssim \ml G_{a,n}(t)\|u_1\|_{L^2\cap L^1}
\end{align}
for all sufficiently large $t$. The implicit constant in the lower estimate and the time after which it holds may depend on $u_1$.
\end{theorem}

\begin{remark}
For the propagation speed, Theorem~\ref{thm:sharp} requires the assumptions of Theorem~\ref{thm:upper} and an additional symbol estimate for the third derivative of $a$. The weighted integral estimates used in the lower bound follow from \eqref{eq:A1} and \eqref{eq:A2m} with $m=1$, as proved in Lemma~\ref{lem:elementary-symbol-consequences}.
\end{remark}

The polynomial family $a(t)=(1+t)^\alpha$ satisfies \eqref{eq:optical-distance} precisely when $\alpha\geqslant-1$.  If $\alpha>-1$, then $\Lambda(t)\approx(1+t)^{\alpha+1}$ and \eqref{eq:A2m} holds.  At the boundary $\alpha=-1$, the optical distance is $\Lambda(t)=1+\ln(1+t)$.  This speed is the endpoint of the non-integrable polynomial family and is treated separately below.  The condition \eqref{eq:A1} remains valid at this endpoint, whereas the condition \eqref{eq:A2m} fails already for $k=1$ by a logarithmic factor, since
\begin{align*}
\frac{|a'(t)|}{a(t)} =\frac{1}{1+t} \ \ \mbox{but}\ \
\frac{a(t)}{\Lambda(t)} =\frac{1}{(1+t)[1+\ln(1+t)]}.
\end{align*}

\begin{theorem}[Sharp estimate at the endpoint]\label{thm:endpoint}
Let $a(t)=(1+t)^{-1}$.  If $u_1\in L^2\cap L^1$ and $P_{u_1}\neq0$, then the solution to the linear wave equation \eqref{eq:main-problem} satisfies
\begin{align}\label{eq:endpoint-rate}
t(\ln t)^{-\frac{n}{4}}|P_{u_1}| \lesssim \|u(t,\cdot)\|_{L^2} \lesssim t(\ln t)^{-\frac{n}{4}}\|u_1\|_{L^2\cap L^1}
\end{align}
for all sufficiently large $t$. The time after which the estimate holds may depend on $u_1$.
\end{theorem}

\begin{remark}\label{rem:no-weight}
For $u_1\in L^1$ with $P_{u_1}\neq0$, continuity of $\widehat u_1$ at the origin yields a radius $\rho=\rho(u_1)>0$ such that
\begin{align*}
|\widehat u_1(\xi)|\geqslant c_n|P_{u_1}| \ \ \mbox{for}\ \ |\xi|\leqslant\rho.
\end{align*}
This estimate provides the initial-data lower bound used in Theorems~\ref{thm:sharp} and~\ref{thm:endpoint}. For the scalar wave equation with zero initial displacement and $P_{u_1}\neq0$, this observation avoids estimating the Fourier remainder $\widehat u_1(\xi)-\widehat u_1(0)$ in the low-frequency lower bounds of \cite{Ikehata=2023,Chen-Ikehata=2026}.
\end{remark}

\subsection{The sharp rate \texorpdfstring{$\ml G_{a,n}(t)$}{G(a,n,t)}}\label{sec:sharp-rate}

For every $t\geqslant0$, since the two terms under the square root in \eqref{eq:general-rate} are nonnegative, we have
\begin{align*}
\ml G_{a,n}(t)&\approx t[\Lambda(t)]^{-\frac{n}{2}} +\left(\frac{\ml Q_{a,n}(t)}{a(t)}\right)^{\frac{1}{2}}\\
&\approx \max\left\{ t[\Lambda(t)]^{-\frac{n}{2}}, \left(\frac{\ml Q_{a,n}(t)}{a(t)}\right)^{\frac{1}{2}} \right\}.
\end{align*}
\begin{itemize}
\item Under the assumptions of Theorem~\ref{thm:sharp}, the first term comes from the ball
\begin{align*}
\left\{\xi\in\mb R^n:\ |\xi|\Lambda(t)\leqslant\varepsilon\right\}
\end{align*}
for a sufficiently small fixed $\varepsilon>0$. On this ball, the Fourier multiplier is comparable to $t$, and the measure is of order $[\Lambda(t)]^{-n}$. In particular, as $t\to+\infty$,
\begin{align*}
t[\Lambda(t)]^{-\frac{n}{2}}
\begin{cases}
\to+\infty&\text{if}\ \ \Lambda(t)=o\bigl(t^{\frac{2}{n}}\bigr),\\[1mm]
\approx1&\text{if}\ \ \Lambda(t)\approx t^{\frac{2}{n}},\\[1mm]
\to0&\text{if}\ \ t^{\frac{2}{n}}=o\bigl(\Lambda(t)\bigr).
\end{cases}
\end{align*}
	\item The second term comes from the amplitudes propagated beyond the small-frequency region.
For $[\Lambda(t)]^{-1}\leqslant r\leqslant1$, let us define the time $s_r\in[0,t]$ by $r\Lambda(s_r)=1$. The change of variables $r=[\Lambda(s)]^{-1}$ in \eqref{eq:general-Q} gives the exact identity:
	\begin{align*}
\frac{\ml Q_{a,n}(t)}{a(t)} = \int_{\frac{1}{\Lambda(t)}}^1 \left[ \sqrt{\frac{a(s_r)}{a(t)}}(1+s_r) \right]^2 r^{n-1}\,\mathrm dr.
	\end{align*}
The expression in brackets has the form of the amplitude bound in \eqref{eq:middle-K-upper}. The normalization $r\Lambda(s_r)=1$ is used here to write the exact identity, while the proofs use a fixed small constant in place of $1$. Thus the second term of $\ml G_{a,n}(t)$ is the $L^2$ norm of the propagated amplitude in brackets with respect to the radial frequency measure $r^{n-1}\,\mathrm dr$. For the lower estimate, the oscillatory factor is controlled after integration over a frequency interval, as proved in Lemma~\ref{lem:oscillatory-observability}. The integrand defining $\ml Q_{a,n}(t)$ admits another useful expression
	\begin{align*}
\frac{(1+s)^2[a(s)]^2}{[\Lambda(s)]^{n+1}} = \left(\frac{\mathrm d\ln\Lambda(s)}{\mathrm d\ln(1+s)}\right)^2 [\Lambda(s)]^{1-n}.
	\end{align*}
Consequently, $\ml Q_{a,n}(t)$ depends on the logarithmic growth of the optical distance, weighted by $[\Lambda(s)]^{1-n}$.
\end{itemize}

The sharp estimate at the endpoint $a(t)=(1+t)^{-1}$ cannot be obtained by direct substitution in \eqref{eq:general-rate}. In this case, $\Lambda(t)=1+\ln(1+t)$ gives
\begin{align*}
\ml Q_{a,n}(t)&\approx t(\ln t)^{-n-1},\\
\ml G_{a,n}(t)&\approx t(\ln t)^{-\frac{n}{2}},
\end{align*}
as $t\to+\infty$, whereas Theorem~\ref{thm:endpoint} gives $t(\ln t)^{-\frac{n}{4}}$.  At this endpoint, the sharp lower bound is obtained from a ball of small frequencies with radius of order $(\ln t)^{-\frac{1}{2}}$, rather than $[\Lambda(t)]^{-1}$. This change of frequency scale accounts for the different rate.

\subsection{Interpretation of the speed assumptions}\label{sec:speed-assumptions}
 \paragraph{\bfseries The condition \eqref{eq:A1}.} To compare the speed on different time intervals after rescaling, for $T>0$ we introduce the normalized speed
\begin{align*}
b_T(\sigma):=\frac{T}{\Lambda(T)}a(T\sigma)\ \ \mbox{for}\ \ 0<\sigma<1.
\end{align*}
By \eqref{eq:optical-distance}, $\frac{\Lambda(T)}{T}$ is asymptotic to the mean value of $a$ on $(0,T)$ as $T\to+\infty$.  Thus, $b_T$ is the propagation speed on the rescaled interval, normalized by $\frac{\Lambda(T)}{T}$. Indeed,
\begin{align*}
\int_0^1b_T(\sigma)\,\mathrm d\sigma =1-\frac{1}{\Lambda(T)}\to1 \ \ \mbox{as}\ \ T\to+\infty.
\end{align*}
With $s=T\sigma$ and $\varphi(T\sigma)=\psi(\sigma)$, the condition \eqref{eq:A1} is equivalent to
\begin{align*}
\int_0^1[b_T(\sigma)]^2|\psi(\sigma)|^2\,\mathrm d\sigma \lesssim \int_0^1|\psi'(\sigma)|^2\,\mathrm d\sigma \ \ \mbox{for}\ \ \psi\in H_0^1(0,1),
\end{align*}
with a constant independent of $T$. Thus, the family of weights $[b_T(\sigma)]^2$ satisfies a uniform weighted Poincar\'e inequality on $(0,1)$.

An integral criterion sufficient for \eqref{eq:A1}, and also necessary under \eqref{eq:A2m} with $m=1$, is
\begin{align}\label{eq:weighted-speed-criterion}
\int_0^T(T-s)s[a(s)]^2\,\mathrm ds \lesssim T[\Lambda(T)]^2 \ \ \mbox{for}\ \ T>0.
\end{align}
For $\varphi\in H_0^1(0,T)$, the boundary conditions give
\begin{align*}
\varphi(s)=\frac{T-s}{T}\int_0^s\varphi'(r)\,\mathrm dr-\frac{s}{T}\int_s^T\varphi'(r)\,\mathrm dr.
\end{align*}
By the Cauchy--Schwarz inequality, we obtain
\begin{align*}
|\varphi(s)|^2 &\leqslant \left[\frac{s(T-s)^2}{T^2}+\frac{(T-s)s^2}{T^2}\right]\int_0^T|\varphi'(r)|^2\,\mathrm dr\\
&=\frac{s(T-s)}{T}\int_0^T|\varphi'(r)|^2\,\mathrm dr.
\end{align*}
Multiplying by $[a(s)]^2$ and integrating over $[0,T]$, we derive
\begin{align*}
\int_0^T[a(s)]^2|\varphi(s)|^2\,\mathrm ds
&\leqslant\frac{1}{T}\int_0^T(T-s)s[a(s)]^2\,\mathrm ds\int_0^T|\varphi'(r)|^2\,\mathrm dr\\
&\lesssim[\Lambda(T)]^2\int_0^T|\varphi'(r)|^2\,\mathrm dr,
\end{align*}
which proves \eqref{eq:A1}. The converse implication under \eqref{eq:A2m} with $m=1$ follows from Lemma~\ref{lem:elementary-symbol-consequences}.

In terms of the normalized speeds $b_T$, the estimate \eqref{eq:weighted-speed-criterion} reads
\begin{align*}
\sup_{T>0}\int_0^1\sigma(1-\sigma)[b_T(\sigma)]^2\,\mathrm d\sigma<+\infty.
\end{align*}
Replacing $\sigma(1-\sigma)$ by $1$ gives the stronger sufficient criterion
\begin{align*}
T\int_0^T[a(s)]^2\,\mathrm ds\lesssim[\Lambda(T)]^2.
\end{align*}
For $a(t)=(1+t)^\alpha$ with $\alpha\geqslant-1$, direct integration gives, as $T\to+\infty$,
\begin{align*}
\frac{T\int_0^T(1+s)^{2\alpha}\,\mathrm ds}{[\Lambda(T)]^2}\approx
\begin{cases}
1&\text{if}\ \ \alpha>-\frac{1}{2},\\[1mm]
\ln T&\text{if}\ \ \alpha=-\frac{1}{2},\\[1mm]
T^{-2\alpha-1}&\text{if}\ \ -1<\alpha<-\frac{1}{2},\\[1mm]
\frac{T}{(\ln T)^2}&\text{if}\ \ \alpha=-1.
\end{cases}
\end{align*}
Thus the stronger criterion holds precisely for $\alpha>-\frac{1}{2}$, whereas \eqref{eq:weighted-speed-criterion} holds for every $\alpha\geqslant-1$ by the monotonicity argument below.

The estimate \eqref{eq:weighted-speed-criterion}, and hence the condition \eqref{eq:A1}, holds for every positive monotone speed.  To see this, we put
\begin{align*}
\Lambda_0(T)&:=\Lambda(T)-1=\int_0^Ta(s)\,\mathrm ds,\\
I_a(T)&:=\int_0^T(T-s)s[a(s)]^2\,\mathrm ds.
\end{align*}
\begin{itemize}
	\item If $a$ is nondecreasing, then
$(T-s)a(s)\leqslant\Lambda_0(T)-\Lambda_0(s)$ and
	\begin{align*}
I_a(T) &\leqslant T\int_0^Ta(s)[\Lambda_0(T)-\Lambda_0(s)]\,\mathrm ds=\frac{T}{2}[\Lambda_0(T)]^2.
	\end{align*}
	\item If $a$ is nonincreasing, then $sa(s)\leqslant\Lambda_0(s)$ and
	\begin{align*}
I_a(T) &\leqslant T\int_0^Ta(s)\Lambda_0(s)\,\mathrm ds=\frac{T}{2}[\Lambda_0(T)]^2.
	\end{align*}
\end{itemize}
Here, we used $\Lambda_0'(s)=a(s)$ and $\Lambda_0(0)=0$ in both cases. Since $\Lambda_0(T)\leqslant\Lambda(T)$, the estimate \eqref{eq:weighted-speed-criterion} follows in both cases. The same conclusion holds for speeds of the form
\begin{align*}
a(t)=\lambda(t)b(t) \ \ \mbox{with}\ \ 0<c\leqslant b(t)\leqslant C,
\end{align*}
where $\lambda$ is positive and monotone. Thus, the condition \eqref{eq:A1} permits bounded relative oscillations about a monotone speed. If $\Lambda_\lambda(t):=1+\int_0^t\lambda(s)\,\mathrm ds$, direct comparison in \eqref{eq:general-Q} and \eqref{eq:general-rate} gives
\begin{align*}
\Lambda(t)\approx\Lambda_\lambda(t) \ \ \mbox{and}\ \ \ml G_{a,n}(t)\approx\ml G_{\lambda,n}(t).
\end{align*}
The derivative estimates must still be verified for the product $\lambda b$. The speed in \eqref{eq:nonmonotone-speed} shows that the assumptions of Theorem~\ref{thm:sharp} also permit speeds outside this class.

\medskip \paragraph{\bfseries The condition \eqref{eq:A2m}.} We introduce the following local time scale to describe the relative growth of the optical distance:
\begin{align*}
h(t):=\frac{\Lambda(t)}{a(t)}.
\end{align*}
Because of $\Lambda'(t)=a(t)$, one has
\begin{align*}
[h(t)]^{-1}=\frac{a(t)}{\Lambda(t)}=\frac{\mathrm d}{\mathrm dt}\ln\Lambda(t).
\end{align*}
The condition \eqref{eq:A2m} then reads
\begin{align*}
|a^{(k)}(t)| \lesssim a(t)[h(t)]^{-k} \ \ \mbox{for}\ \ k\in\{1,\ldots,m\}.
\end{align*}
Each additional time derivative of $a$ is measured by one factor of $[h(t)]^{-1}$. Symbol estimates of this type, in which the derivatives of the propagation speed are measured against an effective time scale, are standard in the analysis of time-dependent hyperbolic equations \cite{Reissig-Yagdjian=2000-increasing,Reissig-Yagdjian=2000-oscillations,Reissig=2004,Hirosawa=2007,Hirosawa-Wirth=2009,Aslan-Reissig=2022}. Here, the effective scale is fixed by the optical distance itself.  For the polynomial, logarithmically modified, and log-periodically oscillating speeds considered in Section~\ref{sec:examples} with $\alpha>-1$, one has $h(t)\approx1+t$.  For the exponential-type speeds considered below, $h(t)\approx(1+t)^{1-\beta}$.  Thus, \eqref{eq:A2m} does not require monotonicity, but it prevents order-one relative variations on time scales much shorter than $h(t)$. The first symbol estimate also gives $h(t)\leqslant H(1+t)$ for some $H>0$, as proved in \eqref{eq:h-linear-upper-first}. Consequently,
\begin{align*}
\ln\Lambda(t)=\int_0^t\frac{\mathrm ds}{h(s)}\geqslant\frac{1}{H}\ln(1+t)
\ \ \mbox{and}\ \ 
\Lambda(t)\geqslant(1+t)^{\frac{1}{H}}.
\end{align*}

\begin{remark}\label{rem:change-time}
The optical distance also gives a change of time relating \eqref{eq:main-problem} to a wave equation with a first-order time derivative. Since $a>0$ and $\Lambda(t)\to+\infty$, the change of variables
\begin{align*}
s:=\Lambda(t)-1 \ \ \mbox{and}\ \ t=t(s):=\Lambda^{-1}(1+s)
\end{align*}
is well defined for $s\geqslant0$. We set
\begin{align*}
v(s,x):=u\bigl(t(s),x\bigr).
\end{align*}
Since $\partial_t=a(t)\partial_s$, we obtain
\begin{align*}
v_{ss}-\Delta v+\beta(s)v_s=0 \ \ \mbox{with}\ \ \beta(s):=\frac{a'\bigl(t(s)\bigr)}{\bigl[a\bigl(t(s)\bigr)\bigr]^2},
\end{align*}
and the initial data are $v(0,x)=0$ and $v_s(0,x)=\frac{u_1(x)}{a(0)}$. Under the condition \eqref{eq:A2m}, differentiation with $\partial_s=\bigl[a\bigl(t(s)\bigr)\bigr]^{-1}\partial_t$ gives
\begin{align*}
|\partial_s^j\beta(s)|\lesssim(1+s)^{-j-1} \ \ \mbox{for}\ \ j\in\{0,\ldots,m-1\}.
\end{align*}
The coefficient $\beta$ has the sign of $a'$, so it need not be nonnegative. We retain the original time variable in the estimates to express the rate directly in terms of $a$ and $\Lambda$.
\end{remark}

\section{The general upper estimate}

By applying the partial Fourier transform with respect to the spatial variables, we obtain
\begin{align}\label{eq:kernel-representation}
\widehat u(t,\xi)=K(t,|\xi|)\widehat u_1(\xi),
\end{align}
where $r:=|\xi|$ and the real-valued kernel $K=K(t,r)$ solves
\begin{align}\label{eq:kernel-equation}
\begin{cases}
K_{tt}+[a(t)]^2r^2K=0,&r\geqslant 0,\ t>0,\\
K(0,r)=0,\ \  K_t(0,r)=1,&r\geqslant0.
\end{cases}
\end{align}
Since $a$ is positive and belongs to $\ml C^1\bigl([0,+\infty)\bigr)$, it is bounded away from zero on every compact time interval.  Standard well-posedness theory for linear strictly hyperbolic equations \cite{Ikawa=2000,Ebert-Reissig=2018} gives a unique energy solution on every finite time interval. To identify the solution class associated with \eqref{eq:kernel-representation}, we use the scalar energy
\begin{align*}
\ml E(t,r):=|K_t(t,r)|^2+[a(t)]^2r^2|K(t,r)|^2.
\end{align*}
The kernel equation yields
\begin{align*}
\partial_t\ml E(t,r)=2a(t)a'(t)r^2|K(t,r)|^2 \leqslant2\frac{|a'(t)|}{a(t)}\ml E(t,r).
\end{align*}
Since $\ml E(0,r)=1$, Gronwall's inequality gives
\begin{align*}
|K_t(t,r)|^2+r^2|K(t,r)|^2\leqslant C_T \ \ \mbox{for}\ \ 0\leqslant t\leqslant T\ \ \mbox{and}\ \ r\geqslant0.
\end{align*}
Together with $|K(t,r)|\leqslant\int_0^t|K_t(s,r)|\,\mathrm ds$, this gives
\begin{align*}
\sup_{0\leqslant t\leqslant T}\bigl((1+r^2)|K(t,r)|^2+|K_t(t,r)|^2\bigr)\leqslant C_T \ \ \mbox{for}\ \ r\geqslant0.
\end{align*}
For $u_1\in L^2$, the continuity of $K$ and $K_t$ in time, together with this bound and dominated convergence, yields
\begin{align*}
u\in\ml C\bigl([0,+\infty),H^1\bigr) \cap\ml C^1\bigl([0,+\infty),L^2\bigr).
\end{align*}

The frequency decomposition and the entrance time used below are illustrated in Figure~\ref{fig:frequency-zones}.
\begin{figure}[htbp]
\centering
\begin{tikzpicture}[x=1.05cm,y=1cm,font=\small]
\draw[->] (0,0)--(10.5,0)node[right]{$t$};
\draw[->] (0,0)--(0,4.7)node[left]{$r$};
\node[below left] at(0,0){$0$};
\draw[dashed] (0,3.4)--(10,3.4);
\node[left] at(0,3.4){$\varepsilon$};
\draw[thick,domain=0:10,samples=100]plot(\x,{3.4/(1+0.5*\x)});
\node[fill=white,inner sep=2pt] at(7.5,1.12){$r=\frac{\varepsilon}{\Lambda(t)}$};
\node at(4.7,3.85){High frequency: $r\geqslant\varepsilon$};
\node at(5.7,2.7){Middle frequency: $\frac{\varepsilon}{\Lambda(t)}\leqslant r\leqslant\varepsilon$};
\node at(6.0,0.24){Low frequency: $r\Lambda(t)\leqslant\varepsilon$};
\draw[dashed] (0,1.6)node[left]{$r_0$}--(2.25,1.6);
\draw[dashed] (2.25,1.6)--(2.25,0)node[below]{$t_{r_0}$};
\fill(2.25,1.6)circle(2pt);
\draw[->,thick] (2.25,1.6)--(9.5,1.6);
\node[above] at(6.2,1.6){fixed $r_0$: increasing time};
\draw[decorate,decoration={brace,amplitude=5pt}] (10.1,4.3)--(10.1,0.58);
\node[rotate=90,above]at(10.35,2.45){hyperbolic zone};
\end{tikzpicture}
\caption{Schematic frequency decomposition.}
\label{fig:frequency-zones}
\end{figure}

 The boundaries are $r=\varepsilon$ and $r=\frac{\varepsilon}{\Lambda(t)}$. For $0<r_0<\varepsilon$, the entrance time satisfies $r_0\Lambda(t_{r_0})=\varepsilon$.

\subsection{The low-frequency zone}

\begin{lemma}\label{lem:low-upper}
Assume that \eqref{eq:A1} and \eqref{eq:A2m} with $m=1$ hold.  There exists $\varepsilon>0$ such that, for $t>0$ and $r\Lambda(t)\leqslant\varepsilon$, one has
\begin{align}\label{eq:low-upper-K}
0<K(t,r)&\leqslant t,\\
|K_t(t,r)|&\label{eq:low-upper-Kt} \lesssim 1+\frac{t\,a(t)}{\Lambda(t)}.
\end{align}
\end{lemma}

\begin{proof}
Let us fix $(t,r)$ with $t>0$ and $r\Lambda(t)\leqslant\varepsilon$.  Since $K(t,0)=t$, it remains to consider $r>0$.  We choose $T>t$ such that
\begin{align*}
\Lambda(T)=2\Lambda(t).
\end{align*}
Suppose that $K(\cdot,r)$ has a first positive zero $T_0\leqslant T$. Multiplying \eqref{eq:kernel-equation} by $K$ and integrating over $(0,T_0)$, we obtain
\begin{align*}
\int_0^{T_0}|K_t(s,r)|^2\,\mathrm ds =r^2\int_0^{T_0}[a(s)]^2|K(s,r)|^2\,\mathrm ds.
\end{align*}
The condition \eqref{eq:A1} then yields
\begin{align*}
\int_0^{T_0}|K_t(s,r)|^2\,\mathrm ds &\leqslant C r^2[\Lambda(T_0)]^2 \int_0^{T_0}|K_t(s,r)|^2\,\mathrm ds\\
&\leqslant 4C\varepsilon^2 \int_0^{T_0}|K_t(s,r)|^2\,\mathrm ds.
\end{align*}
By choosing $\varepsilon$ sufficiently small so that $4C\varepsilon^2<1$, we obtain a contradiction. Therefore,
\begin{align*}
K(s,r)>0 \ \ \mbox{for}\ \ 0<s\leqslant T.
\end{align*}
It follows from \eqref{eq:kernel-equation} that $K(\cdot,r)$ is concave on $[0,T]$.  Since $K(0,r)=0$ and $K_t(0,r)=1$, we obtain
\begin{align*}
0<K(s,r)\leqslant s \ \ \mbox{and}\ \
K_t(s,r)\leqslant1 \ \ \mbox{for}\ \ 0<s\leqslant T.
\end{align*}

By concavity,
\begin{align*}
0<K(T,r) \leqslant K(t,r)+(T-t)K_t(t,r),
\end{align*}
and hence
\begin{align*}
K_t(t,r)>-\frac{K(t,r)}{T-t} \geqslant-\frac{t}{T-t}.
\end{align*}
The condition \eqref{eq:A2m} with $k=1$ gives
\begin{align*}
\left|\ln\frac{a(s)}{a(t)}\right| \lesssim \ln\frac{\Lambda(s)}{\Lambda(t)} \lesssim1 \ \ \mbox{for}\ \ t\leqslant s\leqslant T,
\end{align*}
which shows that $a(s)\approx a(t)$ on $[t,T]$.  Thanks to
\begin{align*}
\int_t^T a(s)\,\mathrm ds =\Lambda(T)-\Lambda(t) =\Lambda(t),
\end{align*}
we have
\begin{align*}
T-t\approx\frac{\Lambda(t)}{a(t)}.
\end{align*}
Combining these estimates proves \eqref{eq:low-upper-K} and \eqref{eq:low-upper-Kt}.
\end{proof}

\subsection{Propagation in the hyperbolic zone}

For $r>0$, the phase variable $\tau:=r\Lambda(t)$ differs from the accumulated phase $r\int_0^t a(s)\,\mathrm ds$ by the constant $r$. To obtain a constant leading matrix after this change of time, we introduce the micro-energy
\begin{align}\label{eq:micro-energy}
W(t,r) :=
\begin{pmatrix}
\sqrt{a(t)}\,rK(t,r)\\[1mm]
[a(t)]^{-\frac{1}{2}}K_t(t,r)
\end{pmatrix}.
\end{align}
To distinguish the two time variables, we set
\begin{align*}
\widetilde W(\tau,r):=W\left(\Lambda^{-1}\left(\frac{\tau}{r}\right),r\right) \ \ \mbox{for}\ \ r>0\ \ \mbox{and}\ \ \tau\geqslant r.
\end{align*}
All $\tau$-derivatives are taken with $r$ fixed. Differentiating the two components in \eqref{eq:micro-energy}, we obtain
\begin{align*}
\partial_tW=a(t)rJW+\frac{a'(t)}{2a(t)}DW,
\end{align*}
where $J$ and $D$ are given below. Since $\partial_t\tau=a(t)r$, this becomes
\begin{align}\label{eq:W-system}
\partial_\tau\widetilde W =\bigl(J+b(\tau,r)D\bigr)\widetilde W,
\end{align}
where
\begin{align*}
J:=
\begin{pmatrix}0&1\\-1&0\end{pmatrix}
\ \ \mbox{and}\ \
D:=
\begin{pmatrix}1&0\\0&-1\end{pmatrix}
\end{align*}
and
\begin{align}\label{eq:def-b}
b(\tau,r) :=\frac{a'(t)}{2[a(t)]^2r} \ \ \mbox{with}\ \ t=\Lambda^{-1}\left(\frac{\tau}{r}\right).
\end{align}

\begin{lemma}\label{lem:hyperbolic-propagation}
Assume that \eqref{eq:A2m} holds with $m=2$.  For every fixed $\varepsilon>0$,
\begin{align}\label{eq:W-upper-propagation}
|\widetilde W(\tau,r)| \lesssim |\widetilde W(\tau_*,r)| \ \ \mbox{for}\ \
\tau\geqslant\tau_*\geqslant\max\{\varepsilon,r\}.
\end{align}
The constant may depend on $\varepsilon$ and $a$, but is independent of $r$, $\tau$, and $\tau_*$.
\end{lemma}

\begin{proof}
In \eqref{eq:W-system}, the bound $b=O(\tau^{-1})$ does not by itself give an integrable coefficient. We first remove the rotation generated by $J$, then use the oscillation of the remaining matrix to obtain an integrable coefficient. Let
\begin{align*}
\mu(t):=\frac{\Lambda(t)\,a'(t)}{[a(t)]^2} \ \ \mbox{so that}\ \ b(\tau,r)=\frac{\mu(t)}{2\tau}.
\end{align*}
By using the condition \eqref{eq:A2m} with $k\in\{1,2\}$, we obtain
\begin{align}\label{eq:b-upper-symbols}
|b(\tau,r)|\lesssim\tau^{-1} \ \ \mbox{and}\ \
|\partial_\tau b(\tau,r)|\lesssim\tau^{-2}.
\end{align}
Indeed, direct differentiation gives
\begin{align*}
\partial_\tau b(\tau,r) =\frac{1}{2a(t)r^2} \left( \frac{a''(t)}{[a(t)]^2} -2\frac{[a'(t)]^2}{[a(t)]^3} \right).
\end{align*}

Let us set $R(\tau):=\exp(\tau J)$ and write $\widetilde W=RY$.  Then
\begin{align*}
Y_\tau=b(\tau,r)M(\tau)Y \ \ \mbox{where}\ \
M(\tau):=R(-\tau)DR(\tau).
\end{align*}
The entries of $M$ have zero mean over one period, and $M$ has the bounded primitive
\begin{align*}
N(\tau) :=\frac{1}{2}
\begin{pmatrix}
\sin(2\tau)&-\cos(2\tau)\\
-\cos(2\tau)&-\sin(2\tau)
\end{pmatrix}
\ \ \mbox{and}\ \ N'(\tau)=M(\tau).
\end{align*}
Choose $\tau_0>\varepsilon$ sufficiently large that, for $\tau\geqslant\max\{\tau_0,r\}$,
\begin{align*}
\|b(\tau,r)N(\tau)\|\leqslant\frac{1}{2},\ \ 
\|I+bN\|\leqslant\frac{3}{2}\ \ \mbox{and}\ \ 
\|(I+bN)^{-1}\|\leqslant2.
\end{align*}
We set $Y=(I+bN)Z$. A direct calculation gives
\begin{align*}
Z_\tau =(I+bN)^{-1}\bigl(-\partial_\tau b\,N+b^2MN\bigr)Z.
\end{align*}
By \eqref{eq:b-upper-symbols}, the operator norm of this coefficient matrix is bounded by $C\tau^{-2}$. Gronwall's inequality gives
\begin{align*}
|Z(\tau,r)|
&\leqslant\exp\left(C\int_{\tau_*}^{\tau}s^{-2}\,\mathrm ds\right)|Z(\tau_*,r)|\\
&\leqslant\exp\left(\frac{C}{\tau_0}\right)|Z(\tau_*,r)|
\end{align*}
for $\tau\geqslant\tau_*\geqslant\max\{\tau_0,r\}$. Since $R$ is orthogonal and the norms of $I+bN$ and its inverse are uniformly bounded, this proves \eqref{eq:W-upper-propagation} on this interval. If $[\max\{\varepsilon,r\},\tau_0]$ is nonempty, we apply Gronwall's inequality there, using $|b(\tau,r)|\lesssim\varepsilon^{-1}$. If $\tau_*<\tau_0<\tau$, we apply the finite-interval estimate from $\tau_*$ to $\tau_0$ and the preceding bound from $\tau_0$ to $\tau$. Both constants are independent of $r$ and $\tau_*$, which proves the assertion.
\end{proof}

Fix $\varepsilon\in(0,1)$ as in Lemma~\ref{lem:low-upper}. For $0<r\leqslant\varepsilon$, let $t_r\geqslant0$ be determined by
\begin{align}\label{eq:entrance-time-upper}
r\Lambda(t_r)=\varepsilon.
\end{align}
For the time scale $h$ introduced in Section~\ref{sec:speed-assumptions}, we have
\begin{align*}
h'(t)=1-\frac{\Lambda(t)a'(t)}{[a(t)]^2}.
\end{align*}
The first symbol estimate in \eqref{eq:A2m} bounds this derivative. Integration yields
\begin{align}\label{eq:h-linear-upper-first}
h(t)\leqslant h(0)+Ct\leqslant H(1+t) \ \ \mbox{for}\ \ t\geqslant0
\end{align}
with some $H>0$.
Since $r\Lambda(t_r)=\varepsilon$, the last bound implies
\begin{align*}
[a(t_r)]^{-\frac{1}{2}}\lesssim_\varepsilon\sqrt{a(t_r)}\,r(1+t_r).
\end{align*}
Applying Lemma~\ref{lem:low-upper} to the two components of $W$, and using continuity when $t_r=0$, we obtain
\begin{align}\label{eq:entrance-energy-upper}
|W(t_r,r)|
&\lesssim\sqrt{a(t_r)}\,rt_r+[a(t_r)]^{-\frac{1}{2}}
\left(1+\frac{t_ra(t_r)}{\Lambda(t_r)}\right)\notag\\
&=\sqrt{a(t_r)}\,rt_r+[a(t_r)]^{-\frac{1}{2}}
+\frac{1}{\varepsilon}\sqrt{a(t_r)}\,rt_r\notag\\
&\lesssim_\varepsilon\sqrt{a(t_r)}\,r(1+t_r).
\end{align}
Since $\widetilde W(\varepsilon,r)=W(t_r,r)$ for $0<r\leqslant\varepsilon$, applying Lemma~\ref{lem:hyperbolic-propagation} to \eqref{eq:entrance-energy-upper} yields
\begin{align}\label{eq:middle-K-upper}
|K(t,r)| \lesssim \sqrt{\frac{a(t_r)}{a(t)}}(1+t_r) \ \ \mbox{for}\ \
\frac{\varepsilon}{\Lambda(t)}\leqslant r\leqslant\varepsilon.
\end{align}

For $r\geqslant\varepsilon$, the initial value of $\tau$ is $\tau_*=r\Lambda(0)=r$, and the initial vector in the original time variable is
\begin{align*}
W(0,r)=
\begin{pmatrix}0\\{[a(0)]^{-\frac{1}{2}}}\end{pmatrix}.
\end{align*}
Since $\widetilde W(r,r)=W(0,r)$, Lemma~\ref{lem:hyperbolic-propagation} gives
\begin{align}\label{eq:high-K-upper}
|K(t,r)| \lesssim \frac{1}{r\sqrt{a(t)}} \ \ \mbox{for}\ \ r\geqslant\varepsilon.
\end{align}

Combining \eqref{eq:low-upper-K}, \eqref{eq:middle-K-upper}, and \eqref{eq:high-K-upper}, we obtain the following bounds.

\begin{lemma}
Under the assumptions of Theorem~\ref{thm:upper}, with $\varepsilon$ fixed as above, the following estimate holds for all $t\geqslant0$ and $r\geqslant0$:
\begin{align}\label{eq:kernel-upper-three-zones}
|K(t,r)| \lesssim
\begin{cases}
t&\text{if}\ \ 0\leqslant r\leqslant\frac{\varepsilon}{\Lambda(t)},\\[2mm]
\displaystyle \sqrt{\frac{a(t_r)}{a(t)}}(1+t_r) &\text{if}\ \ \frac{\varepsilon}{\Lambda(t)}\leqslant r\leqslant\varepsilon,\\[3mm]
\displaystyle \frac{1}{r\sqrt{a(t)}}&\text{if}\ \ r\geqslant\varepsilon,
\end{cases}
\end{align}
where $t_r$ is defined by \eqref{eq:entrance-time-upper}.
\end{lemma}

\subsection{Proof of Theorem~\ref{thm:upper}}

By applying Plancherel's theorem to \eqref{eq:kernel-representation}, we obtain
\begin{align*}
\|u(t,\cdot)\|_{L^2}^2 =\int_{\mb R^n}|K(t,|\xi|)|^2|\widehat u_1(\xi)|^2\,\mathrm d\xi.
\end{align*}
In the low-frequency zone, \eqref{eq:kernel-upper-three-zones} gives
\begin{align}\label{eq:J-low-upper}
\int_{|\xi|\leqslant\frac{\varepsilon}{\Lambda(t)}} |K(t,|\xi|)|^2|\widehat u_1(\xi)|^2\,\mathrm d\xi \lesssim t^2[\Lambda(t)]^{-n}\|u_1\|_{L^1}^2.
\end{align}

For the middle-frequency zone, we use $\|\widehat u_1\|_{L^\infty}\lesssim\|u_1\|_{L^1}$ and polar coordinates, and set $s=t_r$. Since
\begin{align*}
r=\frac{\varepsilon}{\Lambda(s)} \ \ \mbox{and}\ \
|\mathrm dr| =\frac{\varepsilon a(s)}{[\Lambda(s)]^2}\,\mathrm ds,
\end{align*}
the endpoints $r=\frac{\varepsilon}{\Lambda(t)}$ and $r=\varepsilon$ correspond to $s=t$ and $s=0$, respectively. Hence
\begin{align}\label{eq:J-middle-upper}
\int_{\frac{\varepsilon}{\Lambda(t)}\leqslant|\xi|\leqslant\varepsilon}|K(t,|\xi|)|^2|\widehat u_1(\xi)|^2\,\mathrm d\xi&\lesssim\frac{\|u_1\|_{L^1}^2}{a(t)}\int_{\frac{\varepsilon}{\Lambda(t)}}^\varepsilon a(t_r)(1+t_r)^2r^{n-1}\,\mathrm dr\notag\\
&=\frac{\varepsilon^n\|u_1\|_{L^1}^2}{a(t)}\int_0^t\frac{(1+s)^2[a(s)]^2}{[\Lambda(s)]^{n+1}}\,\mathrm ds.
\end{align}

Finally, \eqref{eq:high-K-upper} yields
\begin{align}\label{eq:J-high-upper}
\int_{|\xi|\geqslant\varepsilon} |K(t,|\xi|)|^2|\widehat u_1(\xi)|^2\,\mathrm d\xi \lesssim \frac{1}{a(t)}\|u_1\|_{L^2}^2.
\end{align}
For $t\geqslant1$, restricting the integral in \eqref{eq:general-Q} to $[0,1]$ gives a positive lower bound for $\ml Q_{a,n}(t)$.  Therefore, the right-hand side of \eqref{eq:J-high-upper} is bounded by $C \ml Q_{a,n}(t)[a(t)]^{-1}\|u_1\|_{L^2}^2$.  Combining \eqref{eq:J-low-upper}, \eqref{eq:J-middle-upper}, and \eqref{eq:J-high-upper} proves \eqref{eq:general-upper}.

\section{The matching lower estimate}

A lower bound for the full micro-energy does not give a pointwise lower bound for $K$, since its first component oscillates. We therefore control the frequency variation of the amplitude and recover its size after integration, using Lemma~\ref{lem:oscillatory-observability}. To this end, we first compare $K(t,r)$ with the zero-frequency solution $t$, then propagate the resulting estimates into the hyperbolic zone.

\subsection{The Volterra approximation at the entrance time}

We first derive the weighted integral estimates needed for the Volterra operator.

\begin{lemma}\label{lem:elementary-symbol-consequences}
Suppose that \eqref{eq:A1} and \eqref{eq:A2m} with $m=1$ hold. Then, for $t\geqslant0$,
\begin{align}\label{eq:h-linear-upper}
\frac{\Lambda(t)}{a(t)}\lesssim1+t
\end{align}
and
\begin{align}\label{eq:first-a-moment}
\int_0^t s[a(s)]^2\,\mathrm ds\lesssim t\,a(t)\Lambda(t).
\end{align}
Moreover, the estimate \eqref{eq:weighted-speed-criterion} holds for every $T>0$.
\end{lemma}

\begin{proof}
By \eqref{eq:h-linear-upper-first}, we have $h(t)\leqslant H(1+t)$, which proves \eqref{eq:h-linear-upper}. With $\delta:=H^{-1}>0$, it follows that
\begin{align}\label{eq:Lambda-scale-comparison}
\frac{\Lambda(s)}{\Lambda(t)}
=\exp\left(-\int_s^t\frac{\mathrm dq}{h(q)}\right)
\leqslant\left(\frac{1+s}{1+t}\right)^\delta \ \ \mbox{for}\ \ 0\leqslant s\leqslant t.
\end{align}

To prove \eqref{eq:first-a-moment}, let $t\geqslant2$ and split the integral at $\frac{t}{2}$. For $R\geqslant1$, we take
\begin{align*}
\psi_R(s):=\min\{s,2R-s\} \ \ \mbox{for}\ \ 0\leqslant s\leqslant2R.
\end{align*}
Then $\psi_R\in H_0^1(0,2R)$, $\psi_R(s)\geqslant\frac{R}{2}$ on $[\frac{R}{2},R]$, and
\begin{align*}
\int_0^{2R}|\psi_R'(s)|^2\,\mathrm ds=2R.
\end{align*}
By the condition \eqref{eq:A1}, we obtain
\begin{align*}
\int_{\frac{R}{2}}^R s[a(s)]^2\,\mathrm ds
&\leqslant\frac{4}{R}\int_0^{2R}[a(s)]^2|\psi_R(s)|^2\,\mathrm ds
\lesssim[\Lambda(2R)]^2.
\end{align*}
Choose $J\geqslant1$ so that $1\leqslant2^{-J}t<2$ and apply this estimate with $R=2^{-j}t$ for $j\in\{1,\ldots,J\}$. The remaining interval is contained in $[0,1]$. By \eqref{eq:Lambda-scale-comparison}, we have
\begin{align*}
\int_0^{\frac{t}{2}}s[a(s)]^2\,\mathrm ds
&\lesssim1+\sum_{j=1}^J[\Lambda(2^{1-j}t)]^2\\
&\lesssim[\Lambda(t)]^2\left(1+\sum_{j=1}^J2^{-2\delta j}\right)
\lesssim[\Lambda(t)]^2.
\end{align*}

For the integral over $[\frac{t}{2},t]$, choose $T_*>t$ so that $\Lambda(T_*)=2\Lambda(t)$. The condition \eqref{eq:A2m} with $m=1$ gives
\begin{align*}
\left|\ln\frac{a(s)}{a(t)}\right|
\lesssim\int_t^s\frac{a(q)}{\Lambda(q)}\,\mathrm dq
\leqslant C\ln2 \ \ \mbox{for}\ \ t\leqslant s\leqslant T_*.
\end{align*}
Hence $a(s)\approx a(t)$ on $[t,T_*]$ and
\begin{align}\label{eq:optical-doubling-length}
T_*-t\approx\frac{\Lambda(t)}{a(t)}=h(t).
\end{align}
We choose the right-hand transition on the optical-doubling interval $[t,T_*]$, whose length is comparable to $h(t)$ by \eqref{eq:optical-doubling-length}. More precisely, we set
\begin{align*}
\chi_t(s):=\begin{cases}
\frac{2s}{t}&\mbox{if}\ \ 0\leqslant s\leqslant \frac{t}{2},\\
1&\mbox{if}\ \ \frac{t}{2}\leqslant s\leqslant t,\\
\frac{T_*-s}{T_*-t}&\mbox{if}\ \ t\leqslant s\leqslant T_*.
\end{cases}
\end{align*}
Then $\chi_t\in H_0^1(0,T_*)$, and \eqref{eq:h-linear-upper} and \eqref{eq:optical-doubling-length} give
\begin{align*}
\int_0^{T_*}|\chi_t'(s)|^2\,\mathrm ds
=\frac{2}{t}+\frac{1}{T_*-t}\lesssim\frac{1}{h(t)}.
\end{align*}
Applying the condition \eqref{eq:A1} on $(0,T_*)$, we derive
\begin{align*}
\int_{\frac{t}{2}}^t[a(s)]^2\,\mathrm ds
&\leqslant\int_0^{T_*}[a(s)]^2|\chi_t(s)|^2\,\mathrm ds\\
&\lesssim\frac{[\Lambda(T_*)]^2}{h(t)}
\approx a(t)\Lambda(t).
\end{align*}
Combining the two parts and using \eqref{eq:h-linear-upper}, we obtain
\begin{align*}
\int_0^t s[a(s)]^2\,\mathrm ds
\lesssim[\Lambda(t)]^2+t\,a(t)\Lambda(t)
\lesssim t\,a(t)\Lambda(t).
\end{align*}
Since $a$ is positive and continuous, the same estimate holds on $[0,2]$ after changing the constant. This proves \eqref{eq:first-a-moment}.

Finally, by changing the order of integration and using \eqref{eq:first-a-moment}, we get
\begin{align*}
\int_0^T(T-s)s[a(s)]^2\,\mathrm ds
&=\int_0^T\int_0^\tau s[a(s)]^2\,\mathrm ds\,\mathrm d\tau\\
&\lesssim\int_0^T\tau a(\tau)\Lambda(\tau)\,\mathrm d\tau\\
&\leqslant\frac{T}{2}\bigl([\Lambda(T)]^2-1\bigr)\lesssim T[\Lambda(T)]^2,
\end{align*}
which proves \eqref{eq:weighted-speed-criterion}.
\end{proof}

The solution at $r=0$ is $K(t,0)=t$. To measure the error relative to this function, for $T>0$ we use the Banach space $X_T$ of continuous functions $f$ on $[0,T]$ satisfying $f(0)=0$ and having finite norm
\begin{align*}
\|f\|_{X_T}:=\sup_{0<s\leqslant T}\frac{|f(s)|}{s}.
\end{align*}
Integrating the kernel equation twice leads to the operator
\begin{align*}
(\Phi_r f)(t) :=r^2\int_0^t(t-s)[a(s)]^2f(s)\,\mathrm ds.
\end{align*}
By Lemma~\ref{lem:elementary-symbol-consequences}, for $0<t\leqslant T$ we have
\begin{align*}
\frac{|(\Phi_r f)(t)|}{t}
&\leqslant\frac{r^2}{t}\|f\|_{X_T}\int_0^t(t-s)s[a(s)]^2\,\mathrm ds\\
&\lesssim r^2[\Lambda(t)]^2\|f\|_{X_T}
\leqslant r^2[\Lambda(T)]^2\|f\|_{X_T}.
\end{align*}
Taking the supremum over $0<t\leqslant T$, we obtain
\begin{align}\label{eq:Phi-bound}
\|\Phi_r f\|_{X_T} \leqslant C_\Phi r^2[\Lambda(T)]^2\|f\|_{X_T}.
\end{align}
Here, we used \eqref{eq:weighted-speed-criterion} at time $t$ and $\Lambda(t)\leqslant\Lambda(T)$ for $t\leqslant T$. The kernel equation is equivalent to
\begin{align}\label{eq:Volterra-K}
(I+\Phi_r)K=e \ \ \mbox{where}\ \ e(t):=t.
\end{align}

\begin{lemma}[Entrance data]\label{lem:entrance-data}
Assume that \eqref{eq:A1} and \eqref{eq:A2m} with $m=1$ hold. The constant $\varepsilon$ in Lemma~\ref{lem:low-upper} can be chosen so small that, whenever $r\Lambda(T)\leqslant\varepsilon$, we have
\begin{align}\label{eq:low-expansion-quantitative}
K(T,r) =T+O\bigl(T r^2[\Lambda(T)]^2\bigr)
\end{align}
and
\begin{align}\label{eq:Kr-low-bound}
\|K_r(\cdot,r)\|_{X_T} \lesssim r[\Lambda(T)]^2.
\end{align}
For $0<r<\varepsilon$, let $t_r$ be the entrance time defined in \eqref{eq:entrance-time-upper}. To describe the size of the micro-energy at $t_r$, set
\begin{align*}
B(r):=\sqrt{a(t_r)}\,r(1+t_r).
\end{align*}
For $W$ defined in \eqref{eq:micro-energy}, if $t_r\geqslant1$, then
\begin{align}
|W(t_r,r)|&\approx B(r),\label{eq:entrance-W-size}\\
r\left|\frac{\mathrm d}{\mathrm dr}W(t_r,r)\right|&\lesssim B(r).\label{eq:entrance-W-derivative}
\end{align}
The derivative in \eqref{eq:entrance-W-derivative} is taken along the curve $r\Lambda(t_r)=\varepsilon$, with $t_r$ depending on $r$.
\end{lemma}

\begin{proof}
By \eqref{eq:Phi-bound}, we decrease $\varepsilon$ if necessary so that $C_\Phi\varepsilon^2\leqslant\frac{1}{4}$. For $r\Lambda(T)\leqslant\varepsilon$, the Neumann series in \eqref{eq:Volterra-K} gives
\begin{align*}
\|K-e\|_{X_T}
&\leqslant\frac{C_\Phi r^2[\Lambda(T)]^2}{1-C_\Phi r^2[\Lambda(T)]^2}
\leqslant\frac{1}{3},\\
\|K-e\|_{X_T}&\lesssim r^2[\Lambda(T)]^2.
\end{align*}
This proves \eqref{eq:low-expansion-quantitative}.  Differentiating \eqref{eq:Volterra-K} with respect to $r$ gives
\begin{align*}
\bigl((I+\Phi_r)K_r\bigr)(t) =-2r\int_0^t(t-s)[a(s)]^2K(s,r)\,\mathrm ds.
\end{align*}
The right-hand side has $X_T$ norm bounded by $Cr[\Lambda(T)]^2\|K\|_{X_T}$. Since $\|K\|_{X_T}\lesssim1$ and $(I+\Phi_r)^{-1}$ is uniformly bounded, we obtain \eqref{eq:Kr-low-bound}.

In the remainder of the proof, all functions are evaluated at $(t_r,r)$ unless another argument is displayed.  By using \eqref{eq:low-expansion-quantitative} and \eqref{eq:Kr-low-bound}, we obtain
\begin{align}\label{eq:K-at-entrance}
K=t_r+O(\varepsilon^2t_r) \ \ \mbox{and}\ \
r|K_r|\lesssim t_r.
\end{align}
Differentiating the Volterra equation in time and using Lemma~\ref{lem:elementary-symbol-consequences}, we obtain
\begin{align}\label{eq:Kt-at-entrance}
|K_t-1| &\leqslant r^2\int_0^{t_r}[a(s)]^2|K(s,r)|\,\mathrm ds\lesssim \varepsilon^2\frac{t_r a(t_r)}{\Lambda(t_r)}.
\end{align}
Differentiating once more with respect to $r$, we have
\begin{align*}
K_{tr}(t,r)=-2r\int_0^t[a(s)]^2K(s,r)\,\mathrm ds-r^2\int_0^t[a(s)]^2K_r(s,r)\,\mathrm ds.
\end{align*}
Using \eqref{eq:Kr-low-bound}, $|K(s,r)|\lesssim s$, and \eqref{eq:first-a-moment}, we obtain at $t=t_r$
\begin{align}\label{eq:Ktr-at-entrance}
r|K_{tr}|&\lesssim r^2\bigl(1+r^2[\Lambda(t_r)]^2\bigr)\int_0^{t_r}s[a(s)]^2\,\mathrm ds\notag\\
&\lesssim\varepsilon^2\frac{t_r a(t_r)}{\Lambda(t_r)}
\leqslant\varepsilon^2\left(1+\frac{t_r a(t_r)}{\Lambda(t_r)}\right).
\end{align}
By \eqref{eq:Kt-at-entrance}, \eqref{eq:h-linear-upper}, and $r\Lambda(t_r)=\varepsilon$, the second component of $W$ satisfies
\begin{align*}
[a(t_r)]^{-\frac{1}{2}}|K_t(t_r,r)|
\lesssim[a(t_r)]^{-\frac{1}{2}}\left(1+\varepsilon^2\frac{t_ra(t_r)}{\Lambda(t_r)}\right)
\lesssim_\varepsilon B(r).
\end{align*}
The preceding Neumann-series bound gives $K(t_r,r)\geqslant\frac{2}{3}t_r$. Since $t_r\geqslant1$, the first component satisfies
\begin{align*}
|W(t_r,r)|\geqslant\sqrt{a(t_r)}\,rK(t_r,r)
\geqslant\frac{2}{3}\sqrt{a(t_r)}\,rt_r\geqslant\frac{1}{3}B(r).
\end{align*}
This proves \eqref{eq:entrance-W-size}.

It remains to differentiate the two components of $W$. Differentiating $r\Lambda(t_r)=\varepsilon$ gives
\begin{align}\label{eq:sr-prime}
rt_r'=-\frac{\Lambda(t_r)}{a(t_r)}.
\end{align}
For brevity, write $a_r:=a(t_r)$ and $\Lambda_r:=\Lambda(t_r)$. Using \eqref{eq:sr-prime}, we obtain
\begin{align*}
r\frac{\mathrm d}{\mathrm dr}\bigl(\sqrt{a_r}\,rK(t_r,r)\bigr)&=\sqrt{a_r}\left\{rK+r^2K_r+r^2t_r'K_t
+\frac{a'(t_r)}{2a_r}(rt_r')rK\right\},\\
r\frac{\mathrm d}{\mathrm dr}\bigl(a_r^{-\frac{1}{2}}K_t(t_r,r)\bigr)&=a_r^{-\frac{1}{2}}\left\{rK_{tr}+rt_r'K_{tt}
-\frac{a'(t_r)}{2a_r}(rt_r')K_t\right\}.
\end{align*}
We estimate these terms by \eqref{eq:K-at-entrance}, \eqref{eq:Kt-at-entrance}, \eqref{eq:Ktr-at-entrance}, and the kernel identity $K_{tt}=-a_r^2r^2K$. The condition \eqref{eq:A2m} with $m=1$ and \eqref{eq:h-linear-upper} give
\begin{align}\label{eq:entrance-auxiliary-bounds}
\left|\frac{a'(t_r)}{a_r}rt_r'\right|\lesssim1 \ \ \mbox{and}\ \
|r^2t_r'|=\frac{\varepsilon}{a_r} \ \ \mbox{and}\ \
B(r)\gtrsim_\varepsilon a_r^{-\frac{1}{2}}.
\end{align}
The last inequality follows from $\frac{\Lambda_r}{a_r}\lesssim1+t_r$ and $r\Lambda_r=\varepsilon$. The terms in the first component satisfy
\begin{align*}
\sqrt{a_r}\left(|rK|+|r^2K_r|+\left|\frac{a'(t_r)}{a_r}(rt_r')rK\right|\right)&\lesssim B(r),\\
\sqrt{a_r}\,|r^2t_r'K_t|
&\lesssim\varepsilon a_r^{-\frac{1}{2}}\left(1+\varepsilon^2\frac{t_ra_r}{\Lambda_r}\right)
\lesssim B(r).
\end{align*}
For the second component, \eqref{eq:Ktr-at-entrance} and \eqref{eq:entrance-auxiliary-bounds} yield
\begin{align*}
a_r^{-\frac{1}{2}}\left(|rK_{tr}|+\left|\frac{a'(t_r)}{a_r}(rt_r')K_t\right|\right)&\lesssim a_r^{-\frac{1}{2}}\left(1+\varepsilon^2\frac{t_ra_r}{\Lambda_r}\right)
\lesssim B(r),\\
a_r^{-\frac{1}{2}}|rt_r'K_{tt}|
&=\varepsilon\sqrt{a_r}\,r|K|\lesssim B(r).
\end{align*}
This proves \eqref{eq:entrance-W-derivative}.
\end{proof}

\subsection{Frequency-stable propagation}

To propagate the entrance estimates in Lemma~\ref{lem:entrance-data} into the hyperbolic zone, we control the propagator, its inverse, and its first frequency derivative. The derivative bound requires the symbol estimate for $a'''$.

\begin{lemma}[Transfer matrix]\label{lem:transfer-matrix}
Assume that \eqref{eq:A2m} holds with $m=3$.  Fix $\varepsilon>0$.  For $0<r\leqslant\varepsilon$, let $\ml T(\tau,r)$ solve
\begin{align}\label{eq:Y-system-lower}
\begin{cases}
\partial_\tau\ml T(\tau,r)=b(\tau,r)M(\tau)\ml T(\tau,r),&\tau\geqslant\varepsilon,\\
\ml T(\varepsilon,r)=I,
\end{cases}
\end{align}
where $b$ and $M$ are defined in \eqref{eq:def-b} and in the proof of Lemma~\ref{lem:hyperbolic-propagation}.  Then one has
\begin{align}\label{eq:transfer-bounds}
\|\ml T(\tau,r)\| +\|[\ml T(\tau,r)]^{-1}\| +\|r\partial_r\ml T(\tau,r)\| \lesssim1 \ \ \mbox{for}\ \ \tau\geqslant\varepsilon.
\end{align}
Here, $\partial_r$ is taken with $\tau$ fixed. The constant may depend on $\varepsilon$ and $a$, but is independent of $r$ and $\tau$.
\end{lemma}

\begin{proof}
To differentiate at fixed optical time $\tau$, we use the bounded quantity $\mu$ from the proof of Lemma~\ref{lem:hyperbolic-propagation} and introduce the reciprocal of $h$
\begin{align*}
\eta(t):=\frac{a(t)}{\Lambda(t)}.
\end{align*}
The identities
\begin{align*}
\eta'&=(\mu-1)\eta^2,\\
\mu'&=\frac{a'}{a}+\frac{\Lambda a''}{a^2}-2\frac{\Lambda(a')^2}{a^3}
\end{align*}
give the first derivative bounds below. Differentiating the second identity once more, we obtain
\begin{align*}
\mu''&=\frac{\Lambda a'''}{a^2}+\frac{2a''}{a}-3\left(\frac{a'}{a}\right)^2-\frac{6\Lambda a'a''}{a^3}+\frac{6\Lambda(a')^3}{a^4}.
\end{align*}
The last four terms are $O(\eta^2)$ under the first two symbol estimates. Hence the first three symbol estimates imply
\begin{align}\label{eq:eta-mu-symbols}
|\eta'|\lesssim\eta^2 \ \ \mbox{and}\ \
|\mu|\lesssim1 \ \ \mbox{and}\ \
|\mu'|\lesssim\eta \ \ \mbox{and}\ \
|\mu''|\lesssim\eta^2.
\end{align}
From $\tau=r\Lambda(t)$, we obtain
\begin{align*}
\partial_\tau t=\frac{1}{\tau\eta(t)} \ \ \mbox{and}\ \
r\partial_rt=-\frac{1}{\eta(t)}.
\end{align*}
Since $b=\frac{\mu}{2\tau}$, \eqref{eq:eta-mu-symbols} gives
\begin{align}\label{eq:b-full-symbols}
|b|+|r\partial_rb|\lesssim\tau^{-1} \ \ \mbox{and}\ \
|\partial_\tau b|+|r\partial_r\partial_\tau b| \lesssim\tau^{-2}.
\end{align}
Indeed, differentiating $b$ with respect to $r$ and $\tau$, we obtain
\begin{align*}
r\partial_rb =-\frac{1}{2\tau}\frac{\mu'}{\eta} \ \ \mbox{and}\ \
\partial_\tau b =\frac{1}{2\tau^2}\left(\frac{\mu'}{\eta}-\mu\right).
\end{align*}
Differentiating the second identity with respect to $r$ uses
\begin{align*}
\left(\frac{\mu'}{\eta}\right)' =\frac{\mu''}{\eta}-\frac{\mu'\eta'}{\eta^2} =O(\eta).
\end{align*}
The bound for $r\partial_r\partial_\tau b$ therefore uses the symbol estimate for $a'''$.

Choose $\tau_0>\varepsilon$ sufficiently large that
\begin{align*}
\sup_{0<r\leqslant\varepsilon}\|b(\tau,r)N(\tau)\|\leqslant\frac{1}{2}
\ \ \mbox{for}\ \ \tau\geqslant\tau_0,
\end{align*}
where $N$ is the bounded primitive of $M$ used in Lemma~\ref{lem:hyperbolic-propagation}. On $[\varepsilon,\tau_0]$, set $\ml S_T:=r\partial_r\ml T$. Differentiating \eqref{eq:Y-system-lower} at fixed $\tau$ gives
\begin{align*}
\begin{cases}
\partial_\tau\ml S_T=bM\ml S_T+(r\partial_rb)M\ml T,\\
\ml S_T(\varepsilon,r)=0.
\end{cases}
\end{align*}
Applying \eqref{eq:b-full-symbols} and Gronwall's inequality, we obtain uniform bounds for $\ml T$, $\ml T^{-1}$, and $\ml S_T$ on $[\varepsilon,\tau_0]$.

For the remaining interval, we write
\begin{align*}
Y=(I+bN)Z \ \ \mbox{for}\ \ \tau\geqslant\tau_0.
\end{align*}
Then
\begin{align}\label{eq:Z-lower-system}
Z_\tau=\ml A(\tau,r)Z \ \ \mbox{where}\ \
\ml A :=(I+bN)^{-1}\bigl(-b_\tau N+b^2MN\bigr).
\end{align}
The choice of $\tau_0$ gives $\|(I+bN)^{-1}\|\leqslant2$. Since $N$ is independent of $r$ at fixed $\tau$,
\begin{align*}
r\partial_r(I+bN)^{-1}
=-(I+bN)^{-1}(r\partial_rb)N(I+bN)^{-1}.
\end{align*}
Together with \eqref{eq:b-full-symbols}, this yields
\begin{align*}
\|\ml A(\tau,r)\| +\|r\partial_r\ml A(\tau,r)\| \lesssim\tau^{-2}.
\end{align*}
Let $\ml U(\tau,r)$ be the fundamental matrix of \eqref{eq:Z-lower-system} with $\ml U(\tau_0,r)=I$. Its inverse satisfies
\begin{align*}
\partial_\tau(\ml U^{-1})=-\ml U^{-1}\ml A.
\end{align*}
Since $\|\ml A(\tau,r)\|\lesssim\tau^{-2}$, Gronwall's inequality applied to both equations gives
\begin{align*}
\|\ml U(\tau,r)\|+\|[\ml U(\tau,r)]^{-1}\| \lesssim1 \ \ \mbox{for}\ \ \tau\geqslant\tau_0.
\end{align*}
For $\ml U$, set $\ml S_U:=r\partial_r\ml U$, which solves
\begin{align*}
\begin{cases}
\partial_\tau\ml S_U=\ml A\ml S_U+(r\partial_r\ml A)\ml U,\\
\ml S_U(\tau_0,r)=0.
\end{cases}
\end{align*}
Variation of constants gives
\begin{align*}
\ml S_U(\tau,r)=\ml U(\tau,r)\int_{\tau_0}^{\tau}[\ml U(\sigma,r)]^{-1}
(r\partial_r\ml A(\sigma,r))\ml U(\sigma,r)\,\mathrm d\sigma.
\end{align*}
The bounds for $\ml U$ and its inverse, together with $\|r\partial_r\ml A\|\lesssim\tau^{-2}$, yield $\|\ml S_U(\tau,r)\|\lesssim1$. The matrices $I+bN$ and $(I+bN)^{-1}$, as well as their logarithmic frequency derivatives, are uniformly bounded for $\tau\geqslant\tau_0$. The fundamental matrices on the two intervals satisfy
\begin{align*}
\ml T(\tau,r) &=\bigl(I+b(\tau,r)N(\tau)\bigr)\ml U(\tau,r) \bigl(I+b(\tau_0,r)N(\tau_0)\bigr)^{-1}\ml T(\tau_0,r).
\end{align*}
Each factor, its inverse, and its logarithmic frequency derivative are bounded.  Hence \eqref{eq:transfer-bounds} holds also on $[\tau_0,+\infty)$, which proves the lemma.
\end{proof}

To separate the rotation from the amplitude, for $0<r<\varepsilon$ and $r\Lambda(t)\geqslant\varepsilon$ we define
\begin{align}\label{eq:def-V-amplitude}
V(t,r) :=\ml T\bigl(r\Lambda(t),r\bigr)R(-\varepsilon)W(t_r,r).
\end{align}
Here, $t_r$ is determined by \eqref{eq:entrance-time-upper}. Using $\widetilde W=RY$ and $\widetilde W(\varepsilon,r)=W(t_r,r)$, we obtain the exact identity
\begin{align}\label{eq:W-rotated-amplitude}
W(t,r)=R\bigl(r\Lambda(t)\bigr)V(t,r).
\end{align}

\begin{lemma}[Renormalized amplitude]\label{lem:renormalized-amplitude}
Assume that $a$ satisfies the speed assumptions of Theorem~\ref{thm:sharp}, and fix $\varepsilon$ as in Lemma~\ref{lem:entrance-data}. There exists $\rho\in(0,\varepsilon)$ such that
\begin{align}\label{eq:V-size-derivative}
|V(t,r)|\approx B(r) \ \ \mbox{and}\ \
r|\partial_rV(t,r)|\lesssim|V(t,r)|
\end{align}
whenever $\frac{\varepsilon}{\Lambda(t)}\leqslant r\leqslant\rho$. The constants are independent of $t$ and $r$.
\end{lemma}

\begin{proof}
Choose $\rho\in(0,\varepsilon)$ with $\rho\leqslant\frac{\varepsilon}{\Lambda(1)}$. Then $t_r\geqslant1$ for $0<r\leqslant\rho$. Applying Lemma~\ref{lem:entrance-data} and Lemma~\ref{lem:transfer-matrix}, in particular the bound for $\ml T^{-1}$, we obtain the first estimate in \eqref{eq:V-size-derivative}.

Write $\tau=r\Lambda(t)$ and denote the rotated entrance data by
\begin{align*}
E(r):=R(-\varepsilon)W(t_r,r).
\end{align*}
Differentiating \eqref{eq:def-V-amplitude} at fixed $t$ gives
\begin{align}\label{eq:V-chain-rule}
r\partial_rV(t,r) =\bigl(\tau\partial_\tau\ml T(\tau,r) +r\partial_r\ml T(\tau,r)\bigr)E(r) +\ml T(\tau,r)rE'(r),
\end{align}
where $\partial_r\ml T(\tau,r)$ is evaluated with $\tau$ fixed. By \eqref{eq:Y-system-lower}, we have
\begin{align*}
\tau\partial_\tau\ml T(\tau,r) =\tau b(\tau,r)M(\tau)\ml T(\tau,r),
\end{align*}
which is uniformly bounded by \eqref{eq:b-full-symbols} and \eqref{eq:transfer-bounds}.  Applying \eqref{eq:transfer-bounds} and \eqref{eq:entrance-W-derivative} to \eqref{eq:V-chain-rule}, we obtain
\begin{align*}
r|\partial_rV(t,r)|\lesssim B(r)\approx|V(t,r)|,
\end{align*}
and the proof is complete.
\end{proof}

\subsection{Oscillatory observability}

The first component in \eqref{eq:W-rotated-amplitude} may vanish at individual frequencies even though $|V|$ is bounded from below. Under the derivative bound in \eqref{eq:V-size-derivative}, integration over a frequency interval gives a lower estimate for its weighted $L^2$ norm. In the following lemma, $L$ will be $\Lambda(t)$. The weight $r^{n-3}$ comes from the polar measure $r^{n-1}\,\mathrm dr$ after division by $r^2$ in \eqref{eq:micro-energy}.

\begin{lemma}\label{lem:oscillatory-observability}
Let $V=(V_1,V_2)\in\ml C^1\bigl([R,\rho],\mb R^2\bigr)$ satisfy
\begin{align*}
|V(r)|>0 \ \ \mbox{and}\ \ r|V'(r)|\leqslant C_0|V(r)|.
\end{align*}
For every $n\geqslant1$, there exists $N_0=N_0(n,C_0)$ such that, if $L>0$, $N\geqslant N_0$, $R=\frac{N}{L}$, and $2R<\rho$, then
\begin{align}\label{eq:oscillatory-observability}
\int_R^\rho |V_1(r)\cos(Lr)+V_2(r)\sin(Lr)|^2r^{n-3}\,\mathrm dr\gtrsim \int_R^\rho|V(r)|^2r^{n-3}\,\mathrm dr.
\end{align}
The implicit constant depends only on $n$ and $C_0$.
\end{lemma}

\begin{proof}
Set $p(r):=|V(r)|^2r^{n-3}$.  Expanding the square on the left-hand side of \eqref{eq:oscillatory-observability} gives
\begin{align}\label{eq:oscillatory-expansion}
|V_1(r)\cos(Lr)+V_2(r)\sin(Lr)|^2r^{n-3}=\frac{1}{2}p(r)+A_c(r)\cos(2Lr)+A_s(r)\sin(2Lr),
\end{align}
where
\begin{align*}
A_c(r) :=\frac{1}{2}\bigl([V_1(r)]^2-[V_2(r)]^2\bigr)r^{n-3} \ \ \mbox{and}\ \
A_s(r):=V_1(r)V_2(r)r^{n-3}.
\end{align*}
The hypothesis on $V$ gives
\begin{align*}
|A_c(r)|+|A_s(r)|\lesssim p(r) \ \ \mbox{and}\ \
|A_c'(r)|+|A_s'(r)|\lesssim\frac{p(r)}{r}.
\end{align*}
Integration by parts therefore gives
\begin{align}\label{eq:osc-error}
\left|\int_R^\rho\bigl(A_c(r)\cos(2Lr)+A_s(r)\sin(2Lr)\bigr)\,\mathrm dr\right|\lesssim\frac{1}{L}\left(p(R)+p(\rho)+\int_R^\rho\frac{p(r)}{r}\,\mathrm dr\right).
\end{align}
Since $p$ is positive, we can differentiate its logarithm to obtain
\begin{align*}
\left|\frac{p'(r)}{p(r)}\right|
=\left|\frac{2V(r)\cdot V'(r)}{|V(r)|^2}+\frac{n-3}{r}\right|
\leqslant\frac{2C_0+|n-3|}{r}.
\end{align*}
Integrating between $r,s\in[R,\rho]$, we obtain
\begin{align*}
\left|\ln\frac{p(s)}{p(r)}\right| \leqslant \bigl(2C_0+|n-3|\bigr) \left|\ln\frac{s}{r}\right|.
\end{align*}
Thus, the values of $p$ are comparable on each of the intervals $[R,2R]$ and $[\frac{\rho}{2},\rho]$.  Since $2R<\rho$, these comparisons give
\begin{align*}
p(R) &\lesssim\frac{1}{R}\int_R^{2R}p(r)\,\mathrm dr \ \ \mbox{as well as}\ \ p(\rho) \lesssim\frac{1}{\rho}\int_{\frac{\rho}{2}}^{\rho}p(r)\,\mathrm dr,
\end{align*}
and
\begin{align*}
\int_R^\rho\frac{p(r)}{r}\,\mathrm dr \leqslant\frac{1}{R}\int_R^\rho p(r)\,\mathrm dr.
\end{align*}
Using $LR=N$, we bound the left-hand side of \eqref{eq:osc-error} by
\begin{align*}
\frac{C}{N}\int_R^\rho p(r)\,\mathrm dr.
\end{align*}
Choosing $N_0$ so that $\frac{C}{N_0}<\frac{1}{4}$ and using \eqref{eq:oscillatory-expansion} proves the claim.
\end{proof}

\subsection{Proof of Theorem~\ref{thm:sharp}}

The upper estimate follows from Theorem~\ref{thm:upper}.  It remains to derive the lower bound.

Fix $\varepsilon$ as in Lemma~\ref{lem:entrance-data}. By Remark~\ref{rem:no-weight}, after decreasing $\rho$ in Lemma~\ref{lem:renormalized-amplitude}, we may assume that
\begin{align}\label{eq:Fourier-nondegeneracy-general}
|\widehat u_1(\xi)| \geqslant c_n|P_{u_1}| \ \ \mbox{for}\ \ |\xi|\leqslant\rho.
\end{align}
Take $0<\varepsilon_0<\varepsilon$ so small that \eqref{eq:low-expansion-quantitative} gives
\begin{align*}
K(t,r)\geqslant\frac{t}{2} \ \ \mbox{for}\ \ r\Lambda(t)\leqslant\varepsilon_0.
\end{align*}
For sufficiently large $t$, it follows from Plancherel's identity and \eqref{eq:Fourier-nondegeneracy-general} that
\begin{align}\label{eq:very-low-lower}
\|u(t,\cdot)\|_{L^2}^2 \gtrsim |P_{u_1}|^2t^2[\Lambda(t)]^{-n}.
\end{align}

We next estimate the contribution from propagated frequencies.  Let us choose $N\geqslant N_0$ as in Lemma~\ref{lem:oscillatory-observability}, with $N>\varepsilon$.  The parameters $\varepsilon$, $\rho$, and $N$ are now fixed independently of $t$.  The choice of $\rho$ may depend on $u_1$.  For sufficiently large $t$, we put
\begin{align*}
R_t:=\frac{N}{\Lambda(t)}<\frac{\rho}{2}.
\end{align*}
The first component of \eqref{eq:W-rotated-amplitude} is
\begin{align*}
\sqrt{a(t)}\,rK(t,r) =V_1(t,r)\cos\bigl(r\Lambda(t)\bigr) +V_2(t,r)\sin\bigl(r\Lambda(t)\bigr).
\end{align*}
Using polar coordinates, \eqref{eq:Fourier-nondegeneracy-general}, Lemma~\ref{lem:oscillatory-observability}, and Lemma~\ref{lem:renormalized-amplitude}, we obtain
\begin{align}\label{eq:middle-lower-r}
\|u(t,\cdot)\|_{L^2}^2 &\gtrsim \frac{|P_{u_1}|^2}{a(t)} \int_{R_t}^{\rho}|V(t,r)|^2r^{n-3}\,\mathrm dr\notag\\
&\gtrsim \frac{|P_{u_1}|^2}{a(t)} \int_{R_t}^{\rho}a(t_r)(1+t_r)^2r^{n-1}\,\mathrm dr.
\end{align}
Let $t_\rho$ be given by $\rho\Lambda(t_\rho)=\varepsilon$, and define $\sigma_N(t)$ by
\begin{align*}
\Lambda\bigl(\sigma_N(t)\bigr) =\frac{\varepsilon}{N}\Lambda(t).
\end{align*}
The change of variables $r=\frac{\varepsilon}{\Lambda(s)}$ transforms \eqref{eq:middle-lower-r} into
\begin{align}\label{eq:middle-lower-s}
\|u(t,\cdot)\|_{L^2}^2 \gtrsim \frac{|P_{u_1}|^2}{a(t)} \int_{t_\rho}^{\sigma_N(t)} \frac{(1+s)^2[a(s)]^2}{[\Lambda(s)]^{n+1}}\,\mathrm ds.
\end{align}

The frequency interval $[R_t,\rho]$ accounts only for entrance times in $[t_\rho,\sigma_N(t)]$. To recover $\ml Q_{a,n}(t)$, we estimate the two omitted time intervals. Put
\begin{align*}
q(s):=\frac{(1+s)^2[a(s)]^2}{[\Lambda(s)]^{n+1}},
\end{align*}
which is positive and continuous. The integral over $[0,t_\rho]$ is finite, while the integral over $[t_\rho,t_\rho+1]$ is strictly positive. Since $t_\rho$ is fixed, we get
\begin{align}\label{eq:early-tail-comparison}
\int_0^{t_\rho}q(s)\,\mathrm ds \lesssim \int_{t_\rho}^{t_\rho+1}q(s)\,\mathrm ds.
\end{align}
For all sufficiently large $t$, the interval on the right-hand side is contained in $[t_\rho,\sigma_N(t)]$. The comparison constant in \eqref{eq:early-tail-comparison} may depend on $\rho$, and hence on $u_1$.

On $[\sigma_N(t),t]$, the ratio $\frac{\Lambda(s)}{\Lambda(t)}$ stays between the fixed constants $\frac{\varepsilon}{N}$ and $1$. The condition \eqref{eq:A2m} with $m=1$ yields
\begin{align*}
\left|\ln\frac{a(s)}{a(t)}\right|
\lesssim\int_s^t\frac{a(\zeta)}{\Lambda(\zeta)}\,\mathrm d\zeta
\leqslant\ln\frac{N}{\varepsilon}.
\end{align*}
Thus $a(s)\approx a(t)$ on this interval. Since
\begin{align*}
\int_{\sigma_N(t)}^t a(s)\,\mathrm ds
=\left(1-\frac{\varepsilon}{N}\right)\Lambda(t),
\end{align*}
we also have $t-\sigma_N(t)\approx\frac{\Lambda(t)}{a(t)}$.
Consequently,
\begin{align}\label{eq:late-tail-comparison}
\frac{1}{a(t)} \int_{\sigma_N(t)}^tq(s)\,\mathrm ds &\lesssim \frac{t^2a(t)}{[\Lambda(t)]^{n+1}} \bigl(t-\sigma_N(t)\bigr)\lesssim t^2[\Lambda(t)]^{-n}.
\end{align}
It follows from \eqref{eq:early-tail-comparison} and \eqref{eq:late-tail-comparison} that
\begin{align*}
\frac{\ml Q_{a,n}(t)}{a(t)} \lesssim \frac{1}{a(t)} \int_{t_\rho}^{\sigma_N(t)}q(s)\,\mathrm ds +t^2[\Lambda(t)]^{-n}.
\end{align*}
Combining this estimate with \eqref{eq:very-low-lower} and \eqref{eq:middle-lower-s}, and decreasing the lower-bound constant if necessary, we obtain
\begin{align*}
\|u(t,\cdot)\|_{L^2}^2 \gtrsim |P_{u_1}|^2[\ml G_{a,n}(t)]^2.
\end{align*}
This proves the lower estimate in \eqref{eq:general-sharp} and completes the proof of Theorem~\ref{thm:sharp}.

\section{The endpoint speed}

As noted before Theorem~\ref{thm:endpoint}, the endpoint speed $a(t)=(1+t)^{-1}$ fails \eqref{eq:A2m} already for $m=1$. Since $\Lambda(t)=1+\ln(1+t)$ and $\frac{a'(t)}{[a(t)]^2}=-1$, the change of time in Remark~\ref{rem:change-time} reduces the kernel equation \eqref{eq:kernel-equation} to an equation with constant coefficients. Set $\theta:=\ln(1+t)$ and define the kernel in logarithmic time by
\begin{align*}
\widetilde K(\theta,r):=K(\mathrm e^\theta-1,r).
\end{align*}
Then \eqref{eq:kernel-equation} becomes
\begin{align}\label{eq:endpoint-constant-ode}
\begin{cases}
	\widetilde K_{\theta\theta}-\widetilde K_\theta+r^2\widetilde K=0,&r\geqslant0,\ \theta>0,\\
\widetilde K(0,r)=0\ \ \mbox{and}\ \ \widetilde K_\theta(0,r)=1,&r\geqslant0.	
\end{cases}
\end{align}
The characteristic roots are
\begin{align*}
\lambda_\pm(r):=\frac{1\pm\sqrt{1-4r^2}}{2}.
\end{align*}
The expansions
\begin{align*}
\lambda_+(r)=1-r^2+O(r^4) \ \ \mbox{and}\ \ \lambda_-(r)=r^2+O(r^4) \ \ \mbox{as}\ \ r\to0
\end{align*}
indicate the frequency scale $r\approx\theta^{-\frac{1}{2}}$, since $\exp(\lambda_+(r)\theta)$ contains the factor $\exp(-r^2\theta)$ to leading order. Solving \eqref{eq:endpoint-constant-ode}, we obtain the real-valued formula
\begin{align}\label{eq:endpoint-kernel-formula}
K(t,r)=
\begin{cases}
\displaystyle \frac{(1+t)^{\frac{1+\sqrt{1-4r^2}}{2}} -(1+t)^{\frac{1-\sqrt{1-4r^2}}{2}}}{\sqrt{1-4r^2}} &\text{if}\ \ 0\leqslant r<\frac{1}{2},\\[4mm]
(1+t)^{\frac{1}{2}}\ln(1+t) &\text{if}\ \ r=\frac{1}{2},\\[2mm]
\displaystyle \frac{2(1+t)^{\frac{1}{2}}}{\sqrt{4r^2-1}} \sin\left(\frac{\sqrt{4r^2-1}}{2}\ln(1+t)\right) &\text{if}\ \ r>\frac{1}{2}.
\end{cases}
\end{align}

\begin{lemma}\label{lem:endpoint-kernel-upper}
There are $\delta\in(0,\frac{1}{2})$ and $c,c_\delta>0$ such that, for all sufficiently large $t$,
\begin{align}\label{eq:endpoint-kernel-upper}
|K(t,r)| \lesssim
\begin{cases}
(1+t)\exp\bigl(-cr^2\ln(1+t)\bigr) &\text{if}\ \ 0\leqslant r\leqslant\delta,\\[1mm]
(1+t)^{1-c_\delta}\ln(1+t) &\text{if}\ \ \delta\leqslant r<\frac{1}{2},\\[1mm]
t^{\frac{1}{2}}\ln t &\text{if}\ \ r\geqslant\frac{1}{2}.
\end{cases}
\end{align}
\end{lemma}

\begin{proof}
Fix $\delta\in(0,\frac{1}{2})$. For $0\leqslant r\leqslant\delta$, we have
\begin{align*}
\lambda_-(r)=\frac{2r^2}{1+\sqrt{1-4r^2}}\geqslant r^2
\ \ \mbox{and}\ \ 
\sqrt{1-4r^2}\geqslant\sqrt{1-4\delta^2}>0.
\end{align*}
Since $\lambda_+=1-\lambda_-$, the first line of \eqref{eq:endpoint-kernel-formula} gives
\begin{align*}
0\leqslant K(t,r)\leqslant
\frac{(1+t)\exp\bigl(-r^2\ln(1+t)\bigr)}{\sqrt{1-4\delta^2}},
\end{align*}
which proves the first estimate in \eqref{eq:endpoint-kernel-upper}.

For $\delta\leqslant r<\frac{1}{2}$, put $c_\delta:=1-\lambda_+(\delta)>0$. Since $\lambda_+$ is decreasing, the mean value theorem gives
\begin{align*}
|K(t,r)|\leqslant(1+t)^{\lambda_+(r)}\ln(1+t)
\leqslant(1+t)^{1-c_\delta}\ln(1+t).
\end{align*}
For $r\geqslant\frac{1}{2}$, the last two lines of \eqref{eq:endpoint-kernel-formula} can be written as
\begin{align*}
K(t,r) =(1+t)^{\frac{1}{2}}\ln(1+t) \frac{\sin\left(\frac{1}{2}\sqrt{4r^2-1}\ln(1+t)\right)}{\frac{1}{2}\sqrt{4r^2-1}\ln(1+t)},
\end{align*}
where the quotient is interpreted as one at $r=\frac{1}{2}$.  The last estimate follows from $|\sin z|\leqslant|z|$.
\end{proof}

\begin{proof}[Proof of Theorem~\ref{thm:endpoint}]
For $r\geqslant\delta$, Lemma~\ref{lem:endpoint-kernel-upper} and
\begin{align*}
(1+t)^{1-c_\delta}\ln(1+t)+ t^{\frac{1}{2}}\ln t \lesssim t(\ln t)^{-\frac{n}{4}}
\end{align*}
yield
\begin{align}\label{eq:endpoint-high-upper}
\|K(t,|\xi|)\widehat u_1(\xi)\|_{L^2(|\xi|\geqslant\delta)} \lesssim t(\ln t)^{-\frac{n}{4}}\|u_1\|_{L^2}.
\end{align}
For $r\leqslant\delta$, the $L^1$--$L^\infty$ Fourier estimate gives
\begin{align}\label{eq:endpoint-low-upper}
\|K(t,|\xi|)\widehat u_1(\xi)\|_{L^2(|\xi|\leqslant\delta)}^2&\lesssim t^2\|u_1\|_{L^1}^2 \int_0^\delta \exp(-2cr^2\ln t)r^{n-1}\,\mathrm dr\notag\\
&\lesssim t^2(\ln t)^{-\frac{n}{2}}\|u_1\|_{L^1}^2.
\end{align}
Here, the last estimate follows from the change of variables $y=r\sqrt{\ln t}$.  Plancherel's identity, \eqref{eq:endpoint-high-upper}, and \eqref{eq:endpoint-low-upper} prove the upper bound in \eqref{eq:endpoint-rate}.

For the lower estimate, fix $\delta_*\in(0,\delta]$. If $\ln(1+t)\geqslant1$ and
\begin{align*}
0\leqslant r\leqslant\frac{\delta_*}{\sqrt{\ln(1+t)}},
\end{align*}
then $r\leqslant\delta$ and $r^2\leqslant\lambda_-(r)\leqslant2r^2$. Therefore,
\begin{align*}
(1+t)^{\lambda_+(r)}&\geqslant(1+t)\exp(-2\delta_*^2),\\
(1+t)^{\lambda_-(r)}&\leqslant\exp(2\delta_*^2).
\end{align*}
For $1+t\geqslant2\exp(4\delta_*^2)$, the numerator in the first line of \eqref{eq:endpoint-kernel-formula} is positive. Since its denominator belongs to $(0,1]$, we obtain
\begin{align}\label{eq:endpoint-pointwise-lower}
K(t,r)&\geqslant(1+t)\exp(-2\delta_*^2)-\exp(2\delta_*^2)\notag\\
&\geqslant\frac{1}{2}\exp(-2\delta_*^2)(1+t).
\end{align}
By continuity of $\widehat u_1$ at the origin, choose $\rho>0$ such that
\begin{align*}
|\widehat u_1(\xi)|\geqslant c_n|P_{u_1}| \ \ \mbox{for}\ \ |\xi|\leqslant\rho.
\end{align*}
For all sufficiently large $t$, one has
\begin{align*}
\frac{\delta_*}{\sqrt{\ln(1+t)}}\leqslant\rho.
\end{align*}
It follows from \eqref{eq:endpoint-pointwise-lower} that
\begin{align*}
\|u(t,\cdot)\|_{L^2}^2 &\gtrsim |P_{u_1}|^2t^2 \int_0^{\frac{\delta_*}{\sqrt{\ln(1+t)}}}r^{n-1}\,\mathrm dr\\
&\gtrsim |P_{u_1}|^2t^2(\ln t)^{-\frac{n}{2}}.
\end{align*}
Taking square roots completes the proof.
\end{proof}

\section{Examples}\label{sec:examples}

In the following examples, let us assume that $u_1\in L^2\cap L^1$ and $P_{u_1}\neq0$. All estimates are understood as $t\to+\infty$. The comparison constants may depend on the fixed initial velocity.

\subsection{Polynomial and logarithmically modified speeds}

The threshold below is the value at which the integral defining $\ml Q_{a,n}$ changes from a power to a logarithm for polynomial speeds. For $n\geqslant2$, we introduce
\begin{align*}
\alpha_{\mathrm c}(n):=-\frac{n-2}{n-1}.
\end{align*}
For $n\geqslant1$ and $\alpha>-1$, define the rate function
\begin{align}\label{eq:R-alpha-n}
\ml R_{\alpha,n}(t) :=
\begin{cases}
t^{1-\frac{n}{2}(\alpha+1)} &\text{if}\ \ n=1\ \ \mbox{or}\ \ n\geqslant2\ \mbox{and}\ \alpha<\alpha_{\mathrm c}(n),\\
t^{\frac{n-2}{2(n-1)}}(\ln t)^{\frac{1}{2}} &\text{if}\ \ n\geqslant2\ \mbox{and}\ \alpha=\alpha_{\mathrm c}(n),\\
t^{-\frac{\alpha}{2}} &\text{if}\ \ n\geqslant2\ \mbox{and}\ \alpha>\alpha_{\mathrm c}(n).
\end{cases}
\end{align}

\begin{description}[leftmargin=0pt,labelindent=0pt,itemsep=1em]
\item[Polynomial speeds]
For $a(t)=(1+t)^\alpha$ with $\alpha>-1$, we obtain the following sharp estimate with the rate function defined in \eqref{eq:R-alpha-n}:
\begin{align}\label{eq:polynomial-rates}
\|u(t,\cdot)\|_{L^2}\approx\ml R_{\alpha,n}(t).
\end{align}
By \eqref{eq:polynomial-rates}, for $n=1$ the norm grows, remains of constant order, or decays according as $\alpha<1$, $\alpha=1$, or $\alpha>1$. For $n\geqslant2$, it grows if $\alpha<0$ and decays if $\alpha>0$. For $\alpha=0$, the speed is $a(t)\equiv1$, and \eqref{eq:polynomial-rates} recovers the classical free-wave rates $t^{\frac{1}{2}}$, $(\ln t)^{\frac{1}{2}}$ and $1$ for $n=1$, $n=2$ and $n\geqslant3$, respectively, as in \cite{Ikehata=2023,Chen-Ikehata=2026}.

\item[Logarithmically modified polynomial speeds]
Let
\begin{align*}
a(t)=(1+t)^\alpha[\ln(\mathrm e+t)]^\beta \ \ \mbox{with}\ \ \alpha>-1\ \ \mbox{and}\ \ \beta\in\mb R.
\end{align*}
If $n=1$, or if $n\geqslant2$ and $\alpha<\alpha_{\mathrm c}(n)$, then
\begin{align}\label{eq:log-modified-subcritical}
\|u(t,\cdot)\|_{L^2} \approx t^{1-\frac{n}{2}(\alpha+1)}(\ln t)^{-\frac{n\beta}{2}}.
\end{align}
If $n\geqslant2$ and $\alpha>\alpha_{\mathrm c}(n)$, then
\begin{align}\label{eq:log-modified-supercritical}
\|u(t,\cdot)\|_{L^2} \approx t^{-\frac{\alpha}{2}}(\ln t)^{-\frac{\beta}{2}}.
\end{align}
For $n\geqslant2$ and $\alpha=\alpha_{\mathrm c}(n)$,
\begin{align}\label{eq:log-modified-critical}
\|u(t,\cdot)\|_{L^2} \approx t^{-\frac{\alpha}{2}}
\begin{cases}
(\ln t)^{\frac{1-n\beta}{2}} &\text{if}\ \ \beta<\frac{1}{n-1},\\[1mm]
(\ln t)^{-\frac{1}{2(n-1)}}(\ln\ln t)^{\frac{1}{2}} &\text{if}\ \ \beta=\frac{1}{n-1},\\[1mm]
(\ln t)^{-\frac{\beta}{2}} &\text{if}\ \ \beta>\frac{1}{n-1}.
\end{cases}
\end{align}
The logarithmic factor gives the following transitions at the indicated polynomial powers:
\begin{itemize}
\item If $n=2$ and $\alpha=0$, the norm grows, remains of constant order, or decays according as $\beta<\frac{1}{2}$, $\beta=\frac{1}{2}$, or $\beta>\frac{1}{2}$.
\item If $n\geqslant3$ and $\alpha=0$, the norm grows, remains of constant order, or decays according as $\beta<0$, $\beta=0$, or $\beta>0$.
\item If $n=1$ and $\alpha=1$, the norm grows, remains of constant order, or decays according as $\beta<0$, $\beta=0$, or $\beta>0$.
\end{itemize}
\end{description}

For these two classes of speeds, we put $\ell(t):=\ln(\mathrm e+t)$ and write $a(t)=(1+t)^\alpha[\ell(t)]^\beta$ with $\alpha>-1$ and $\beta\in\mb R$. Karamata's theorem \cite[Theorem~1.5.11]{Bingham-Goldie-Teugels=1987} gives
\begin{align*}
\Lambda(t)\approx(1+t)^{\alpha+1}[\ell(t)]^\beta.
\end{align*}
Since $2\alpha+1>-1$, the same theorem gives
\begin{align*}
\int_0^T(T-s)s[a(s)]^2\,\mathrm ds
\leqslant T\int_0^T s[a(s)]^2\,\mathrm ds
\lesssim T[\Lambda(T)]^2.
\end{align*}
Direct differentiation also gives
\begin{align*}
	\frac{|a^{(k)}(t)|}{a(t)}\lesssim(1+t)^{-k}\approx\left(\frac{a(t)}{\Lambda(t)}\right)^k \ \ \mbox{for}\ \ k\in\{1,2,3\}.
\end{align*}
These estimates extend to bounded time intervals after changing the constants. Thus, \eqref{eq:weighted-speed-criterion} yields the condition \eqref{eq:A1}, and the condition \eqref{eq:A2m} holds with $m=3$.

Moreover, \eqref{eq:general-Q} yields
\begin{align*}
	\ml Q_{a,n}(t)\approx1+\int_1^t(1+s)^{-(n-1)(\alpha+1)}[\ell(s)]^{-(n-1)\beta}\,\mathrm ds \ \ \mbox{for}\ \ t\geqslant2.
\end{align*}
In the noncritical cases, evaluating this integral gives
\begin{align*}
\ml Q_{a,n}(t)\approx
\begin{cases}
t^{1-(n-1)(\alpha+1)}(\ln t)^{-(n-1)\beta}&\text{if}\ \ n=1\ \ \mbox{or}\ \ n\geqslant2\ \mbox{and}\ \alpha<\alpha_{\mathrm c}(n),\\[1mm]
1&\text{if}\ \ n\geqslant2\ \mbox{and}\ \alpha>\alpha_{\mathrm c}(n).
\end{cases}
\end{align*}
In the first case, the two terms in \eqref{eq:general-rate} have the same order:
\begin{align*}
\frac{\ml Q_{a,n}(t)}{a(t)}\approx t^{2-n(\alpha+1)}(\ln t)^{-n\beta}\approx t^2[\Lambda(t)]^{-n}.
\end{align*}
In the second case, the first term is of lower order, since $2-n(\alpha+1)<-\alpha$, and we have
\begin{align*}
t^2[\Lambda(t)]^{-n}=o\bigl([a(t)]^{-1}\bigr).
\end{align*}
Substitution into \eqref{eq:general-rate} proves \eqref{eq:log-modified-subcritical} and \eqref{eq:log-modified-supercritical} by Theorem~\ref{thm:sharp}. At $\alpha=\alpha_{\mathrm c}(n)$ with $n\geqslant2$, we have $(n-1)(\alpha+1)=1$, and hence
\begin{align*}
\ml Q_{a,n}(t)
&\approx1+\int_2^t\frac{1}{s(\ln s)^{(n-1)\beta}}\,\mathrm ds\approx
\begin{cases}
(\ln t)^{1-(n-1)\beta}&\text{if}\ \ \beta<\frac{1}{n-1},\\[1mm]
\ln\ln t&\text{if}\ \ \beta=\frac{1}{n-1},\\[1mm]
1&\text{if}\ \ \beta>\frac{1}{n-1}.
\end{cases}
\end{align*}
Since $2-n(\alpha+1)=-\alpha$ at this threshold, the ratio of the two terms in \eqref{eq:general-rate} satisfies
\begin{align*}
\frac{t^2a(t)[\Lambda(t)]^{-n}}{\ml Q_{a,n}(t)}\approx
\begin{cases}
(\ln t)^{-1}&\text{if}\ \ \beta<\frac{1}{n-1},\\[1mm]
\frac{1}{\ln t\,\ln\ln t}&\text{if}\ \ \beta=\frac{1}{n-1},\\[1mm]
(\ln t)^{-(n-1)\beta}&\text{if}\ \ \beta>\frac{1}{n-1}.
\end{cases}
\end{align*}
Each expression tends to zero as $t\to+\infty$.
Thus,
\begin{align*}
\ml G_{a,n}(t)\approx\left(\frac{\ml Q_{a,n}(t)}{a(t)}\right)^{\frac{1}{2}},
\end{align*}
which proves \eqref{eq:log-modified-critical} by Theorem~\ref{thm:sharp}. The choice $\beta=0$ gives \eqref{eq:polynomial-rates}.

\subsection{Oscillating speeds}

\begin{description}[leftmargin=0pt,labelindent=0pt,itemsep=1em]
\item[Log-periodically oscillating speeds]
Let
\begin{align*}
a(t)=(1+t)^\alpha \left[2+\sin\bigl(\omega\ln(1+t)\bigr)\right] \ \ \mbox{with}\ \ \alpha>-1\ \ \mbox{and}\ \ \omega\neq0.
\end{align*}
Then
\begin{align}\label{eq:oscillating-polynomial-rate}
\|u(t,\cdot)\|_{L^2}\approx\ml R_{\alpha,n}(t).
\end{align}
The bounded oscillatory factor preserves the rate in \eqref{eq:polynomial-rates}, although the speed need not be monotone.

For the log-periodically oscillating speed, we have
\begin{align*}
	a(t)\approx(1+t)^\alpha \ \ \mbox{and}\ \ \Lambda(t)\approx(1+t)^{\alpha+1}.
\end{align*}
By the comparison argument in Section~\ref{sec:speed-assumptions}, the condition \eqref{eq:A1} holds. Moreover, by direct differentiation, we verify the condition \eqref{eq:A2m} with $m=3$. Comparing the expressions in \eqref{eq:general-Q} and \eqref{eq:general-rate} with those for the polynomial speed, we obtain \eqref{eq:oscillating-polynomial-rate} from Theorem~\ref{thm:sharp}.

\item[An oscillating speed without a monotone comparison]
Consider
\begin{align}\label{eq:nonmonotone-speed}
a(t)=(\mathrm e+t)^{\frac{1}{4}\sin(\ln\ln(\mathrm e+t))}.
\end{align}
This speed satisfies the assumptions of Theorem~\ref{thm:sharp}, but is not comparable to any positive monotone function. For $n\geqslant3$, we obtain
\begin{align}\label{eq:nonmonotone-rate}
\|u(t,\cdot)\|_{L^2}\approx[a(t)]^{-\frac{1}{2}}.
\end{align}
Along suitable time sequences, $a(t)$ tends to zero and to infinity, respectively. Hence
\begin{align*}
\liminf_{t\to+\infty}\|u(t,\cdot)\|_{L^2}=0 \ \ \mbox{and}\ \ \limsup_{t\to+\infty}\|u(t,\cdot)\|_{L^2}=+\infty.
\end{align*}

For the speed in \eqref{eq:nonmonotone-speed}, logarithmic differentiation gives
\begin{align*}
\frac{a'(t)}{a(t)}
=\frac{\sin(\ln\ln(\mathrm e+t))+\cos(\ln\ln(\mathrm e+t))}{4(\mathrm e+t)}.
\end{align*}
With $\kappa:=\frac{\sqrt{2}}{4}<1$, integration yields
\begin{align*}
\left(\frac{\mathrm e+s}{\mathrm e+t}\right)^\kappa
\leqslant\frac{a(s)}{a(t)}
\leqslant\left(\frac{\mathrm e+s}{\mathrm e+t}\right)^{-\kappa}
\ \ \mbox{for}\ \ 0\leqslant s\leqslant t.
\end{align*}
Integrating the upper bound over $[0,t]$ and the lower bound over $[\frac{t}{2},t]$, we obtain
\begin{align}\label{eq:nonmonotone-Lambda}
\Lambda(t)\approx(1+t)a(t).
\end{align}
This also verifies \eqref{eq:optical-distance}, since $a(t)\geqslant(\mathrm e+t)^{-\frac{1}{4}}$. The same comparison gives
\begin{align*}
\int_0^T s[a(s)]^2\,\mathrm ds
&\leqslant[a(T)]^2(\mathrm e+T)^{2\kappa}\int_0^T s(\mathrm e+s)^{-2\kappa}\,\mathrm ds\\
&\lesssim(1+T)^2[a(T)]^2\approx[\Lambda(T)]^2.
\end{align*}
Multiplication by $T$ proves \eqref{eq:weighted-speed-criterion}, and hence the condition \eqref{eq:A1}. Further differentiation of $\frac{a'}{a}$ gives
\begin{align*}
\frac{|a^{(k)}(t)|}{a(t)}\lesssim(1+t)^{-k}
\approx\left(\frac{a(t)}{\Lambda(t)}\right)^k \ \ \mbox{for}\ \ k\in\{1,2,3\}.
\end{align*}
Thus the condition \eqref{eq:A2m} holds with $m=3$. The values of the sine in \eqref{eq:nonmonotone-speed} attain $1$ and $-1$ at arbitrarily large times, so $a$ has subsequences tending to infinity and to zero. A positive nondecreasing function is bounded away from zero for large time, while a positive nonincreasing function is bounded from above. Therefore neither can be comparable to this speed.

For $n\geqslant3$, the bound $a(s)\geqslant(\mathrm e+s)^{-\frac{1}{4}}$ and \eqref{eq:nonmonotone-Lambda} imply
\begin{align*}
\frac{(1+s)^2[a(s)]^2}{[\Lambda(s)]^{n+1}}
\lesssim(1+s)^{-\frac{3}{4}(n-1)}.
\end{align*}
Since $-\frac{3}{4}(n-1)<-1$, we have $\ml Q_{a,n}(t)\approx1$ for $t\geqslant1$. Moreover,
\begin{align*}
t^2a(t)[\Lambda(t)]^{-n}
\lesssim t^{2-n}[a(t)]^{1-n}
\lesssim t^{\frac{7-3n}{4}},
\end{align*}
which tends to zero as $t\to+\infty$ for $n\geqslant3$. Substitution in \eqref{eq:general-rate} proves \eqref{eq:nonmonotone-rate} by Theorem~\ref{thm:sharp}.
\end{description}

\subsection{Exponential-type speeds}

Let
\begin{align*}
a(t)=(1+t)^\gamma \exp\bigl((1+t)^\beta-1\bigr) \ \ \mbox{with}\ \ \beta>0\ \ \mbox{and}\ \ \gamma\in\mb R.
\end{align*}
Then
\begin{align}\label{eq:stretched-exponential-rate}
\|u(t,\cdot)\|_{L^2} \approx
\begin{cases}
t^{\beta+\frac{1}{2}}[a(t)]^{-\frac{1}{2}}&\text{if}\ \ n=1,\\
[a(t)]^{-\frac{1}{2}}&\text{if}\ \ n\geqslant2.
\end{cases}
\end{align}
The estimates in \eqref{eq:stretched-exponential-rate} imply $L^2$ decay of the solution for all $n\geqslant1$. Taking $\beta=1$ and $\gamma=0$ in \eqref{eq:stretched-exponential-rate}, we obtain for $a(t)=\mathrm e^t$
\begin{align*}
\|u(t,\cdot)\|_{L^2}\approx
\begin{cases}
t^{\frac{3}{2}}\mathrm e^{-\frac{t}{2}}&\text{if}\ \ n=1,\\
\mathrm e^{-\frac{t}{2}}&\text{if}\ \ n\geqslant2.
\end{cases}
\end{align*}

To estimate the optical distance for the exponential-type speed, we write $a(t)=\mathrm e^{\phi(t)}$, where
\begin{align*}
\phi(t):=(1+t)^\beta-1+\gamma\ln(1+t).
\end{align*}
For all sufficiently large $t$, one has
\begin{align*}
\phi'(t)\approx(1+t)^{\beta-1}.
\end{align*}
By splitting the integral defining $\Lambda(t)$ at $\frac{t}{2}$, we derive the following estimate:
\begin{align}\label{eq:Lambda-stretched}
\Lambda(t)\approx\frac{a(t)}{\phi'(t)} \approx a(t)(1+t)^{1-\beta}.
\end{align}
Indeed, put $d_T:=(1+T)^{1-\beta}$. For $\frac{T}{2}\leqslant s\leqslant T$,
\begin{align*}
\phi(T)-\phi(s) \gtrsim (T-s)(1+T)^{\beta-1},
\end{align*}
whereas the integral over $[0,\frac{T}{2}]$ can be estimated using the eventual positivity of $\phi'$ and
\begin{align*}
\phi(T)-\phi\left(\frac{T}{2}\right)\gtrsim T^\beta.
\end{align*}
For all sufficiently large $T$, these bounds give
\begin{align*}
1+\int_0^{\frac{T}{2}}a(s)\,\mathrm ds
\lesssim T a(T)\exp(-cT^\beta)
=o\bigl(a(T)d_T\bigr).
\end{align*}
Consequently,
\begin{align*}
\Lambda(T)\lesssim a(T)d_T+a(T)\int_{\frac{T}{2}}^T\exp\left(-c\frac{T-s}{d_T}\right)\,\mathrm ds
\lesssim a(T)d_T.
\end{align*}
Since $\frac{d_T}{T}\to0$, we have $[T-d_T,T]\subset[\frac{T}{2},T]$ for all sufficiently large $T$. On this interval, the upper bound for $\phi'$ gives
\begin{align*}
\phi(T)-\phi(s)&=\int_s^T\phi'(q)\,\mathrm dq
\leqslant C(T-s)(1+T)^{\beta-1}\\
&\leqslant Cd_T(1+T)^{\beta-1}=C.
\end{align*}
Hence
\begin{align*}
\Lambda(T)\geqslant\int_{T-d_T}^T a(s)\,\mathrm ds \geqslant\mathrm e^{-C}a(T)d_T,
\end{align*}
which proves the lower bound in \eqref{eq:Lambda-stretched}. For the weighted integral in \eqref{eq:weighted-speed-criterion}, the same comparison gives
\begin{align*}
\int_0^{\frac{T}{2}}(T-s)s[a(s)]^2\,\mathrm ds
&\lesssim T^3[a(T)]^2\exp(-cT^\beta)
\lesssim T[a(T)]^2d_T^2.
\end{align*}
On the remaining interval, the substitution $\sigma=T-s$ gives
\begin{align*}
\int_{\frac{T}{2}}^T(T-s)s[a(s)]^2\,\mathrm ds
&\lesssim T[a(T)]^2\int_0^{+\infty}\sigma\mathrm e^{-\frac{c\sigma}{d_T}}\,\mathrm d\sigma\\
&\lesssim T[a(T)]^2d_T^2
\approx T[\Lambda(T)]^2.
\end{align*}
Combining the two estimates and adjusting the constant on bounded time intervals, we obtain \eqref{eq:weighted-speed-criterion}, and hence the condition \eqref{eq:A1}.  Moreover, the derivatives of $\phi$ give
\begin{align*}
\frac{|a^{(k)}(t)|}{a(t)} \lesssim (1+t)^{k(\beta-1)} \approx \left(\frac{a(t)}{\Lambda(t)}\right)^k \ \ \mbox{for}\ \ k\in\{1,2,3\},
\end{align*}
after changing the constant on a bounded time interval.  Hence \eqref{eq:A2m} holds with $m=3$. Finally, \eqref{eq:Lambda-stretched} gives
\begin{align*}
\frac{(1+s)^2[a(s)]^2}{[\Lambda(s)]^{n+1}}\approx(1+s)^{2+(n+1)(\beta-1)}[a(s)]^{1-n}
\end{align*}
for all sufficiently large $s$. For $n=1$, this is comparable to $(1+s)^{2\beta}$. For $n\geqslant2$, the factor $[a(s)]^{1-n}$ makes the integral convergent. Therefore,
\begin{align*}
\ml Q_{a,n}(t)\approx\begin{cases}
	(1+t)^{2\beta+1}&\mbox{if}\ \ n=1,\\
	1&\mbox{if}\ \ n\geqslant 2.
\end{cases}
\end{align*}
For comparison with the first term of \eqref{eq:general-rate}, we have
\begin{align*}
t^2a(t)[\Lambda(t)]^{-n}
\lesssim\begin{cases}
t^{\beta+1}&\mbox{if}\ \ n=1,\\
t^{2+n(\beta-1)}[a(t)]^{1-n}&\mbox{if}\ \ n\geqslant2.
\end{cases}
\end{align*}
For $n=1$, division by $\ml Q_{a,1}(t)\approx t^{2\beta+1}$ gives an upper bound of order $t^{-\beta}$, which tends to zero. For $n\geqslant2$, the exponential factor makes the displayed upper bound tend to zero, while $\ml Q_{a,n}(t)\approx1$. Hence
\begin{align*}
t^2[\Lambda(t)]^{-n}=o\left(\frac{\ml Q_{a,n}(t)}{a(t)}\right).
\end{align*}
The second term in \eqref{eq:general-rate} therefore determines the rate, and Theorem~\ref{thm:sharp} proves \eqref{eq:stretched-exponential-rate}.

\raggedbottom
\subsection*{Acknowledgements}
Halit S. Aslan is supported by the S\~ao Paulo Research Foundation (FAPESP), grant No.~2025/24251-0.  Wenhui Chen is supported in part by the National Natural Science Foundation of China, grant No.~12301270, and the Guangdong Basic and Applied Basic Research Foundation, grant No.~2025A1515010240. The authors thank Ryo Ikehata and Michael Reissig for some suggestions in the preparation of this manuscript.


\begin{thebibliography}{99}

  \bibitem{Aslan-Reissig=2022}\newblock H. S. Aslan, M. Reissig.
  \newblock $L^p$--$L^q$ estimates for wave equations with strong time-dependent oscillations.
  \newblock \emph{Hokkaido Math. J.} \textbf{51} (2022), no. 1, 57--106.

  \bibitem{Bingham-Goldie-Teugels=1987}\newblock N. H. Bingham, C. M. Goldie, J. L. Teugels.
  \newblock \emph{Regular Variation}.
  \newblock Cambridge University Press, Cambridge, 1987.

  \bibitem{Chen=2026-Ball}\newblock W. Chen.
  \newblock On a question of Ball for the linear wave equation.
  \newblock Preprint, 2026.

  \bibitem{Chen-Ikehata=2026}\newblock W. Chen, R. Ikehata.
  \newblock Large time behavior for the classical wave equation with different regular data and its applications.
  \newblock \emph{Asymptot. Anal.} \textbf{150} (2026), no. 4, 1638--1659.

  \bibitem{Chen-Takeda=2023}\newblock W. Chen, H. Takeda.
  \newblock Large-time asymptotic behavior for the classical thermoelastic system.
  \newblock \emph{J. Differential Equations} \textbf{377} (2023), 809--848.

  \bibitem{Ebert-Fitriana-Hirosawa=2015}\newblock M. R. Ebert, L. Fitriana, F. Hirosawa.
  \newblock On the energy estimates of the wave equation with time dependent propagation speed asymptotically monotone functions.
  \newblock \emph{J. Math. Anal. Appl.} \textbf{432} (2015), no. 2, 654--677.

  \bibitem{Ebert-Reissig=2018}\newblock M. R. Ebert, M. Reissig.
  \newblock \emph{Methods for Partial Differential Equations: Qualitative Properties of Solutions, Phase Space Analysis, Semilinear Models}.
  \newblock Birkh\"auser, Cham, 2018.

  \bibitem{Galstian=2003}\newblock A. Galstian.
  \newblock $L^p$--$L^q$ decay estimates for the wave equations with exponentially growing speed of propagation.
  \newblock \emph{Appl. Anal.} \textbf{82} (2003), no. 3, 197--214.

  \bibitem{Ghisi-Gobbino=2025}\newblock M. Ghisi, M. Gobbino.
  \newblock Generalized energy conservation for linear wave equations with time-dependent propagation speed.
  \newblock \emph{Math. Ann.} \textbf{392} (2025), no. 1, 1447--1479.

  \bibitem{Hirosawa=2007}\newblock F. Hirosawa.
  \newblock On the asymptotic behavior of the energy for the wave equations with time depending coefficients.
  \newblock \emph{Math. Ann.} \textbf{339} (2007), no. 4, 819--838.

  \bibitem{Hirosawa-Wirth=2009}\newblock F. Hirosawa, J. Wirth.
  \newblock Generalised energy conservation law for wave equations with variable propagation speed.
  \newblock \emph{J. Math. Anal. Appl.} \textbf{358} (2009), no. 1, 56--74.

  \bibitem{Ikawa=2000}\newblock M. Ikawa.
  \newblock \emph{Hyperbolic Partial Differential Equations and Wave Phenomena}.
  \newblock Translations of Mathematical Monographs \textbf{189}, American Mathematical Society, Providence, RI, 2000.

  \bibitem{Ikehata=2023}\newblock R. Ikehata.
  \newblock $L^2$-blowup estimates of the wave equation and its application to local energy decay.
  \newblock \emph{J. Hyperbolic Differ. Equ.} \textbf{20} (2023), no. 1, 259--275.

  \bibitem{Ikehata=2025-wave}\newblock R. Ikehata.
  \newblock $L^2$-boundedness for 1-D wave equations with time variable coefficients.
  \newblock \emph{Arch. Math. (Basel)} \textbf{125} (2025), no. 4, 445--453.

  \bibitem{Lax-Phillips=1989}\newblock P. D. Lax, R. S. Phillips.
  \newblock \emph{Scattering Theory}.
  \newblock Revised ed., Academic Press, Boston, 1989.

  \bibitem{Matsuyama=2006}\newblock T. Matsuyama.
  \newblock Asymptotic behaviour for wave equation with time-dependent coefficients.
  \newblock \emph{Ann. Univ. Ferrara} \textbf{52} (2006), no. 2, 383--393.

  \bibitem{Morawetz=1961}\newblock C. S. Morawetz.
  \newblock The decay of solutions of the exterior initial-boundary value problem for the wave equation.
  \newblock \emph{Comm. Pure Appl. Math.} \textbf{14} (1961), no. 3, 561--568.

  \bibitem{Morawetz-Ralston-Strauss=1977}\newblock C. S. Morawetz, J. V. Ralston, W. A. Strauss.
  \newblock Decay of solutions of the wave equation outside nontrapping obstacles.
  \newblock \emph{Comm. Pure Appl. Math.} \textbf{30} (1977), no. 4, 447--508.

  \bibitem{Peral=1980}\newblock J. C. Peral.
  \newblock $L^p$ estimates for the wave equation.
  \newblock \emph{J. Funct. Anal.} \textbf{36} (1980), no. 1, 114--145.

  \bibitem{Reissig=1997}\newblock M. Reissig.
  \newblock On $L^p$--$L^q$ estimates for solutions of a special weakly hyperbolic equation.
  \newblock In: T.-T. Li (ed.), \emph{Nonlinear Evolution Equations and Infinite-Dimensional Dynamical Systems}.
  \newblock World Scientific, Singapore, 1997, 153--164.

  \bibitem{Reissig=2004}\newblock M. Reissig.
  \newblock $L^p$--$L^q$ decay estimates for wave equations with time-dependent coefficients.
  \newblock \emph{J. Nonlinear Math. Phys.} \textbf{11} (2004), no. 4, 534--548.

  \bibitem{Reissig-Smith=2005}\newblock M. Reissig, J. Smith.
  \newblock $L^p$--$L^q$ estimate for wave equation with bounded time dependent coefficient.
  \newblock \emph{Hokkaido Math. J.} \textbf{34} (2005), no. 3, 541--586.

  \bibitem{Reissig-Yagdjian=2000-increasing}\newblock M. Reissig, K. Yagdjian.
  \newblock $L^p$--$L^q$ decay estimates for the solutions of strictly hyperbolic equations of second order with increasing in time coefficients.
  \newblock \emph{Math. Nachr.} \textbf{214} (2000), 71--104.

  \bibitem{Reissig-Yagdjian=2000-oscillations}\newblock M. Reissig, K. Yagdjian.
  \newblock $L^p$--$L^q$ decay estimates for hyperbolic equations with oscillations in coefficients.
  \newblock \emph{Chinese Ann. Math. Ser. B} \textbf{21} (2000), no. 2, 153--164.

  \bibitem{Strichartz=1970}\newblock R. S. Strichartz.
  \newblock Convolutions with kernels having singularities on a sphere.
  \newblock \emph{Trans. Amer. Math. Soc.} \textbf{148} (1970), 461--471.

  \bibitem{Takeda=2026}\newblock H. Takeda.
  \newblock $L^2$-estimates for the linear elastic waves.
  \newblock \emph{Math. Ann.} \textbf{394} (2026), no. 4, Paper No. 82, 31 pp.

\end{thebibliography}
\end{document}